\documentclass[a4paper,10pt]{article}
\usepackage{ae,lmodern}
\usepackage[english,greek,frenchb,french]{babel}
\usepackage[utf8]{inputenc}  
\usepackage[T1]{fontenc}
\usepackage{url,csquotes}
\usepackage[hidelinks,hyperfootnotes=false]{hyperref}
\usepackage{float}
\usepackage{amsfonts,amssymb,enumerate}
\usepackage{amsmath}
\usepackage{authblk}
\allowdisplaybreaks[1]
\usepackage{amsthm}
\usepackage{graphicx}
\usepackage{bbm}
\usepackage{listings}
\usepackage{titling}
\usepackage{dsfont}
\usepackage{color}
\usepackage{bmpsize}
\newcommand{\E}{\mathbb{E}}
    \newcommand{\Prb}{\mathbb{P}}
	
		\newcommand{\cE}{\mathcal{E}}
		\newcommand{\cF}{\mathcal{F}}
		\newcommand{\cG}{\mathcal{G}}
	
		\newcommand{\cV}{\mathcal{V}}
		\newcommand{\cH}{\mathcal{H}}
	\newcommand{\sR}{\mathbb{R}}
	\newcommand{\sN}{\mathbb{N}}
	\newcommand{\vv}{\overrightarrow{v}}
	
		\newcommand{\ff}{\overrightarrow{f}}
		\newcommand{\sS}{\mathbb{S}}
	\DeclareMathOperator{\Diam}{Diam}
		\DeclareMathOperator{\ext}{ext}
		\DeclareMathOperator{\slab}{slab}

				\DeclareMathOperator{\card}{card}
			\DeclareMathOperator{\hyp}{hyp}
				
				\DeclareMathOperator{\cyl}{cyl}
				\DeclareMathOperator{\Var}{Var}
				\DeclareMathOperator{\flow}{flow}
    \newcommand{\sZ}{\mathbb{Z}}
    \newcommand{\sC}{\mathcal{C}}
    \newcommand{\ep}{\varepsilon}

    \newcommand{\ind}{\mathds{1}}
 \theoremstyle{plain}   
 \newtheorem{thm}{Theorem}
\newtheorem{lem}[thm]{Lemma}

\newtheorem{prop}[thm]{Proposition}
\newtheorem{rk}[thm]{Remark}

\title{Size of a minimal cutset in supercritical first passage percolation \thanks{Research was partially supported by the ANR project PPPP (ANR-16-CE40-0016)}}
\date{}
\author{Barbara Dembin, Marie Théret \thanks{LPSM UMR 8001, Université Paris Diderot, Sorbonne Paris Cité, CNRS, F-75013 Paris, France}}

\begin{document}

 \selectlanguage{english}

\maketitle

\textbf{Abstract:} We consider the standard model of i.i.d. first passage percolation on $\sZ^d$ given a distribution $G$ on $[0,+\infty]$ (including $+\infty$). We suppose that $G(\{0\})>1-p_c(d)$, \textit{i.e.}, the edges of positive passage time are in the subcritical regime of percolation on $\sZ^d$. We consider a cylinder of basis an hyperrectangle of dimension $d-1$ whose sides have length $n$ and of height $h(n)$ with $h(n)$ negligible compared to $n$ (\textit{i.e.}, $h(n)/n\rightarrow 0$ when $n$ goes to infinity). We study the maximal flow from the top to the bottom of this cylinder. We already know that the maximal flow renormalized by $n^{d-1}$ converges towards the flow constant which is null in the case $G(\{0\})>1-p_c(d)$. The study of maximal flow is associated with the study of sets of edges of minimal capacity that cut the top from the bottom of the cylinder. If we denote by $\psi_n$ the minimal cardinal of such a set of edges, we prove here that $\psi_n/n^{d-1}$ converges almost surely towards a constant. \\ 


\section{Introduction}

The model of first passage percolation was first introduced by Hammersley and Welsh \cite{HammersleyWelsh} in 1965 as a model for the spread of a fluid in a porous medium. In this model, mathematicians studied intensively geodesics, \textit{i.e.},  fastest paths between two points in the grid. The study of maximal flows in first passage percolation started later in 1984 in dimension $2$ with an article of Grimmett and Kesten \cite{GrimmettKesten84}. In 1987, Kesten studied maximal flows in dimension $3$ in \cite{Kesten:flows}. The study of maximal flows is associated with the study of random cutsets that can be seen as $(d-1)$-dimensional surfaces. Their study presents more technical difficulties than the study of geodesics. Thus, the interpretation of first passage percolation in term of maximal flows has been less studied. 

Let us consider a large box in $\sZ^d$, to each edge we assign a random i.i.d. capacity with distribution $G$. We interpret this capacity as a rate of flow, \textit{i.e.}, it corresponds to the maximal amount of water that can cross the edge per second. Next, we consider two opposite sides of the box that we call top and bottom. We are interested in the maximal flow that can cross the box from its top to its bottom per second. A first issue is to understand if the maximal flow in the box properly renormalized converges when the size of the box grows to infinity. This question was addressed in \cite{Kesten:flows},  \cite{Rossignol2010} and \cite{Zhang2017} where one can find laws of large numbers and large deviation estimates for this maximal flow when the dimensions of the box grow to infinity under some moments assumptions on the capacities. The maximal flow properly renormalized converges towards the so-called flow constant.  In \cite{flowconstant}, Rossignol and Théret proved the same results without any moment assumption on $G$, they even allow the capacities to take infinite value as long as $G(\{+\infty\})<p_c(d)$ where $p_c(d)$ denotes the critical parameter of i.i.d. bond percolation on $\sZ^d$. We can interpret infinite capacities as a defect of the medium, \textit{i.e.}, there are some edges where the capacities are of bigger order. Moreover, the two authors have shown that the flow constant is continuous with regard to the distribution of the capacities. 

The flow constant is associated with the study of surfaces with minimal capacity. These surfaces must disconnect the top from the bottom of the box in a sense we will precise later. We want to know if the minimal size of these surfaces of minimal capacity grows at the same order as the size of the bottom of the box. When $G(\{0\})<1-p_c(d)$, Zhang proved in \cite{Zhang2017}, under an exponential moment condition, that there exists a constant such that the probability that all the surfaces of minimal capacity are bigger than this constant times the size of the bottom of the cylinder, decays exponentially fast when the size of the box grows to infinity. The main result of this paper is that under the assumption $G(\{0\})>1-p_c(d)$, the minimal size of a surface of minimal capacity divided by the size of the bottom of the cylinder converges towards a constant when the size of the box grows to infinity. 

The rest of the paper is organized as follows. In section \ref{s2}, we give all the necessary definitions and background, we state our main theorem and give the main ideas of the proof. In section \ref{s3}, we define an alternative flow which is more adapted for using subadditive arguments. The proof is made of three steps that correspond to sections \ref{s4}, \ref{s5} and \ref{s6}.

\section{Definition, background and main results}\label{s2}

\subsection{Definition of maximal flows and minimal cutsets}

We keep many notations used in \cite{flowconstant}. We consider the graph $(\sZ ^d,\E ^d)$ where $\E^d$ is the set of edges that link  all the nearest neighbors for the Euclidean norm in $\sZ^d$. We consider a distribution $G$ on $[0,\infty]$. To each edge $e$ in $\E^d$ we assign a random variable $t_G(e)$ with distribution $G$. The variable $t_G(e)$ is called the capacity (or the passage time) of $e$. The family $(t_G(e))_{e\in\E ^d}$ is independent.

Let $\Omega=(V_\Omega,E_\Omega)$ be a finite subgraph of $(\sZ ^d,\E ^d)$. We can see $\Omega$ as a piece of rock through which water can flow. Let $\mathfrak{G}_1$ and $\mathfrak{G}_2$ be two disjoint subsets of $V_\Omega$ representing respectively the sources through which the water can enter and the sinks through which the water exits. 

Let the function $\ff : \E^d\rightarrow \sR^d$ be a possible stream inside $\Omega$ between $\mathfrak{G}_1$ and $\mathfrak{G}_2$. For all $e\in \E^d$, $\|\ff(e)\|_2$ represents the amount of water that flows through $e$ per second and $\ff(e)/\|\ff(e)\|_2$ represents the direction in which the water flows through $e$. If we write $e=\langle x,y \rangle$ where $x,y$ are neighbors in the graph $(\sZ ^d,\E ^d)$, then the unit vector $\ff(e)/\|\ff(e)\|_2$ is either the vector $\overrightarrow{xy}$ or $\overrightarrow{yx}$. We say that our stream $\ff$ inside $\Omega$ from $\mathfrak{G}_1$ to $\mathfrak{G}_2$ is $G$-admissible if and only if it satisfies the following constraints.
\begin{itemize}
\item[$\cdot$] \textit{The node law} : for every vertex $x$ in $V_\Omega\setminus (\mathfrak{G}_1 \cup \mathfrak{G}_2)$, we have
$$\sum_{y\in\sZ^d:\,e =\langle x,y \rangle\in E_\Omega }\|\ff(e)\|_2\left(\ind_{\frac{\ff(e)}{\|\ff(e)\|_2}=\overrightarrow{xy}}-\ind_{\frac{\ff(e)}{\|\ff(e)\|_2}=\overrightarrow{yx}}\right)=0\,,$$ \textit{i.e.}, there is no loss of fluid inside $\Omega$.
\item[$\cdot$] \textit{The capacity constraint}: for every edge $e$ in $E_\Omega$, we have 
$$0\leq \|\ff(e)\|_2\leq t_G(e)\,,$$
\textit{i.e.}, the amount of water that flows through $e$ per second is limited by its capacity $t_G(e)$.
\end{itemize}

Note that as the capacities are random, the set of $G$-admissible streams inside $\Omega$ between $\mathfrak{G}_1$ and $\mathfrak{G}_2$ is also random. For each $G$-admissible stream $\ff$, we define its flow by 
$$\flow(\ff)=\sum_{x\in \mathfrak{G}_1}\sum_{y\in\sZ^d:\,e =\langle x,y \rangle\in E_\Omega }\|\ff(e)\|_2\left(\ind_{\frac{\ff(e)}{\|\ff(e)\|_2}=\overrightarrow{xy}}-\ind_{\frac{\ff(e)}{\|\ff(e)\|_2}=\overrightarrow{yx}}\right)\,.$$
This corresponds to the amount of water that enters in $\Omega$ through $\mathfrak{G}_1$ per second. By the node law, as there is no loss of fluid, $\flow(\ff)$ is also equal to the amount of water that escapes from $\Omega$ through $\mathfrak{G}_2$ per second:

$$\flow(\ff)=\sum_{x\in \mathfrak{G}_2}\sum_{y\in\sZ^d:\,e =\langle x,y \rangle\in E_\Omega }\|\ff(e)\|_2\left(\ind_{\frac{\ff(e)}{\|\ff(e)\|_2}=\overrightarrow{yx}}-\ind_{\frac{\ff(e)}{\|\ff(e)\|_2}=\overrightarrow{xy}}\right)\,.$$

The maximal flow from $\mathfrak{G}_1$ to $\mathfrak{G}_2$ in $\Omega$ for the capacities $(t_G(e))_{e\in\E^d}$, denoted by $\phi_G(\mathfrak{G}_1\rightarrow \mathfrak{G}_2\text{ in } \Omega)$, is the supremum of the flows of all admissible streams through $\Omega$:
$$\phi_G(\mathfrak{G}_1\rightarrow \mathfrak{G}_2\text{ in } \Omega)=\sup\left\{\flow(\ff)\,:\,\begin{array}{c}\ff\text{ is a $G$-admissible stream inside}\\ \text{$\Omega$ between $\mathfrak{G}_1$ and $\mathfrak{G}_2$}\end{array}\right\}\,.$$

Dealing with admissible streams is not so easy, but hopefully we can use an alternative definition of maximal flow which is more convenient. Let $E\subset E_\Omega$ be a set of edges. We say that $E$ cuts $\mathfrak{G}_1$ from $\mathfrak{G}_2$ in $\Omega$ (or is a cutset, for short) if there is no path from $\mathfrak{G}_1$ to $\mathfrak{G}_2$ in $(V_\Omega, E_\Omega \setminus E)$. More precisely, let $\gamma$ be a path from $\mathfrak{G}_1$ to $\mathfrak{G}_2$ in $\Omega$, we can write $\gamma$ as a finite sequence $(v_0,e_1,v_1,\dots,e_n,v_n)$ of vertices $(v_i)_{i=0,\dots,n}\in V_\Omega^{n+1}$ and edges $(e_i)_{i=1,\dots,n}\in E_\Omega ^n$ where $v_0\in  \mathfrak{G}_1$, $v_n\in \mathfrak{G}_2$ and for any $1\leq i \leq n$, $e_i=\langle v_{i-1}, v_i \rangle \in E_\Omega$. Then, $E$ cuts $\mathfrak{G}_1$ from $\mathfrak{G}_2$ in $\Omega$ if for any path $\gamma$ from $\mathfrak{G}_1$ to $\mathfrak{G}_2$ in $\Omega$, we have $\gamma\cap E\neq \emptyset$. Note that $\gamma$ can be seen as a set of edges or a set of vertices and we define $|\gamma|=n$.

We associate with any set of edges $E$ its capacity $T_G(E)$ defined by $T_G(E)=\sum _{e\in E} t_G(e)$. The max-flow min-cut theorem, see \cite{Bollobas}, a result of graph theory, states that 
$$\phi_G(\mathfrak{G}_1\rightarrow \mathfrak{G}_2\text{ in } \Omega)=\min\{T_G(E)\,:\, E \text{ cuts $\mathfrak{G}_1$ from $\mathfrak{G}_2$ in $\Omega$}\}\,.$$
The idea behind this theorem is quite intuitive. When we consider a maximal flow through $\Omega$, some of the edges are jammed. We say that $e$ is jammed if the amount of water that flows through $e$ is equal to the capacity $t_G(e)$. These jammed edges form a cutset, otherwise we would be able to find a path $\gamma$ from $\mathfrak{G}_1$ to $\mathfrak{G}_2$ of non-jammed edges, and we could increase the amount of water that flows through $\gamma$ which contradicts the fact that the flow is maximal. Thus, the flow is always smaller than the capacity of any cutset. It can be proved that the maximal flow is equal to the minimal capacity of a cutset. 

In \cite{Kesten:flows}, Kesten interpreted the study of maximal flow as a higher dimensional version of the classical problem of first passage percolation which is the study of geodesics. A geodesic may be considered as an object of dimension $1$, it is a path with minimal passage time. On the contrary, the maximal flow is associated (via the max-flow min-cut theorem) with cutsets of minimal capacity: those cutsets are objects of dimension $d-1$, that can be seen as surfaces. To better understand the interpretation in term of surfaces, we can associate with each edge $e$ a small plaquette $e^*$. The plaquette $e^*$ is an hypersquare of dimension $d-1$ whose sides have length one and are parallel to the edges of the graphs, such that $e^*$ is normal to $e$ and cuts it in its middle. We associate with the plaquette $e^*$ the same capacity $t_G(e)$ as with the edge $e$. We also define the dual of a set of edge $E$ by $E^*=\{e^*,\,e\in E\}$. Roughly speaking, if the set of edges $E$ cuts $\mathfrak{G}_1$ from $\mathfrak{G}_2$ in $\Omega$, the surface of plaquettes $E^*$ disconnects $\mathfrak{G}_1$ from $\mathfrak{G}_2$ in $\Omega$. Although this interpretation in terms of surfaces seems more intuitive than cutsets, it is really technical to handle, and we will never use it and not even try to give a rigorous definition of a surface of plaquettes. Note that, in dimension $2$, a surface of plaquettes is very similar to a path in the dual graph of $\sZ^2$ and thus the study of minimal cutsets is very similar to the study of geodesics. 

We now consider two specific maximal flows through a cylinder for first passage percolation on $\sZ^d$ where the law of capacities is given by a distribution $G$ such as $G([-\infty,0))=0$ and $G(\{0\})>1-p_c(d)$, \textit{i.e.}, the edges of positive capacity are in the sub-critical regime of percolation on $\sZ^d$. We are interested in the study of cutsets in a cylinder. Among all the minimal cutsets, we are interested with the ones with minimal size. 
Let us first define the maximal flow from the top to the bottom of a cylinder. Let $A$ be a non-degenerate hyperrectangle, \textit{i.e.}, a rectangle of dimension $d-1$ in $\sR^d$. We denote by $\cH^{d-1}$ the Hausdorff measure in dimension $d-1$: for $A=\prod_{i=1}^{d-1}[k_i,l_i]\times \{c\}$ with $k_i<l_i$, $c\in\sR$ we have $\cH^{d-1}(A)=\prod_{i=1}^{d-1}(l_i-k_i)$. Let $\vv$ be one of the two unit vectors normal to $A$. Let $h>0$, we denote by $\cyl(A,h)$ the cylinder of basis $A$ and height $h$ defined by 
$$\cyl(A,h)=\{x+t\vv\, : \,  x\in A,\, t\in[0,h]\}\,.$$
The dependence on $\vv$ is implicit in the notation $\cyl(A,h)$.
We have to define discretized versions of the bottom $B(A,h)$ and the top $T(A,h)$ of the cylinder $\cyl(A,h)$. We define them by 
$$B(A,h):= \left\{x\in\sZ^d\cap\cyl(A,h)\,:\,\begin{array}{c}
\exists y \notin \cyl(A,h),\, \langle x,y \rangle\in\E^d \\\text{ and $\langle x,y \rangle$ intersects } A
\end{array} \right\}$$
and
$$T(A,h):= \left\{x\in\sZ^d\cap\cyl(A,h)\,:\,\begin{array}{c}
\exists y \notin \cyl(A,h),\, \langle x,y \rangle\in\E^d \\\text{ and $\langle x,y \rangle$ intersects } A+h\vv
\end{array} \right\}\,.$$

We denote by $\Phi_G(A,h)$ the maximal flow from the top to the bottom of the cylinder $\cyl(A,h)$ in the direction $\vv$, defined by 
$$\Phi_G(A,h)=\Phi_G(T(A,h)\rightarrow B(A,h) \text{ in } \cyl(A,h))\,.$$

This definition of the flow is not well suited to use ergodic subadditive theorems, because we cannot glue two cutsets from the top to the bottom of two adjacent cylinders together to build a cutset from the top to the bottom of the union of these two cylinders. The reason is that the intersection of a cutset from the top to the bottom of a cylinder with the boundary of the cylinder is totally free. We can fix this issue by introducing another flow through the cylinder for which the subadditivity would be recover. To define this flow, we will first define another version of the cylinder which is more convenient. We denote by $\cyl'(A,h)= \{x+t\vv\, : \,  x\in A,\, t\in[-h,h]\}$. The set $\cyl'(A,h)\setminus A$ has two connected components denoted by $C_1(A,h)$ and $C_2(A,h)$. We have to define a discretized version of the boundaries of these two sets. For $i=1,\,2$, we denote by $C'_i(A,h)$ the discrete boundary of $C_i(A,h)$ defined by 
$$C'_i(A,h)=\{x\in\sZ^d\cap C_i(A,h)\,:\,\exists y\notin \cyl'(A,h),\,\langle x,y\rangle\in\E^d\}\,.$$
We call informally $C'_i(A,h)$, $i=1,2$, the upper and lower half part of the boundary of $\cyl'(A,h)$.
We denote by $\tau_G(A,h)$ the maximal flow from the upper half part to the lower half part of the boundary of the cylinder, \textit{i.e.}, 
$$\tau_G(A,h)=\phi_G(C'_1(A,h)\rightarrow C'_2(A,h)\text{ in } \cyl'(A,h))\,.$$
By the max-flow min-cut theorem, the flow $\tau_G(A,h)$ is equal to the minimal capacity of a set of edges $E$ that cuts $C'_1(A,h)$ from $C'_2(A,h)$ inside the cylinder $\cyl'(A,h)$. If we consider the dual set $E^*$ of $E$, the intersection of $E^*$ with the boundary of the cylinder has to be close to the boundary $\partial A$ of the hyperrectangle $A$. 

\begin{rk}Note that here we will work only with the cylinder $\cyl(A,h)$ whereas the authors of \cite{flowconstant} work mainly with the cylinder $\cyl'(A,h)$.
\end{rk}

\subsection{Background on maximal flows}

The simplest case to study maximal flows is for a straight cylinder, \textit{i.e.}, when $\vv=\overrightarrow{v_0}:=(0,0,\dots,1)$ and $A=A(\overrightarrow{k},\overrightarrow{l})=\prod_{i=1}^{d-1}[k_i,l_i]\times \{0\}$ with $k_i\leq 0<l_i\in\sZ$. In this case, the family of variables $(\tau_G(A(\overrightarrow{k},\overrightarrow{l}),h))_{\overrightarrow{k},\overrightarrow{l}}$ is subadditive since minimal cutsets in adjacent cylinders can be glued together  along the common side of these cylinders. By applying ergodic subadditive theorems in the multi-parameter case (see Krengel and Pyke \cite{KrengelPyke} and Smythe \cite{Smythe}), we obtain the following result.

\begin{prop}\label{prop1} Let $G$ be an integrable probability measure on $[0,+\infty[$, \textit{i.e.}, $\int_{\sR^+}xdG(x)<\infty$. Let $A=\prod_{i=1}^{d-1}[k_i,l_i]\times \{0\}$ with $k_i\leq 0<l_i\in\sZ$. Let $h:\,\sN\rightarrow \sR^+$ such that $\lim_{n\rightarrow \infty}h(n)=+\infty$. Then there exists a constant $\nu_G(\overrightarrow{v_0})$, that does not depend on $A$ and $h$ but depends on $G$ and $d$, such that 
$$\lim_{n\rightarrow \infty}\frac{\tau_G(nA,h(n))}{\cH^{d-1}(nA)}=\nu_G(\overrightarrow{v_0})\text{  a.s. and in $L^1$}.$$
\end{prop}
The constant $\nu_G(\overrightarrow{v_0})$ is called the flow constant.

Next, a natural question to ask is whether we can define a flow constant for any direction. When we consider tilted cylinders, we cannot recover perfect subadditivity because of the discretization of the boundary. Moreover, the use of ergodic subadditive theorems is not possible when the direction $\vv$ we consider is not rational, \textit{i.e.}, when there does not exist an integer $M$ such that $M\vv$ has integer coordinates. Indeed, in that case there exists no vector $\overrightarrow{u}$ normal to $\vv$ such that the model is invariant under the translation of vector $\overrightarrow{u}$. These issues were overcome by Rossignol and Théret in \cite{Rossignol2010} where they proved the following law of large numbers.

\begin{thm}
 Let $G$ be a probability measure on $[0,+\infty[$ such that $G$ is integrable, \textit{i.e.}, $\int_{\sR^+}xdG(x)<\infty$. For any $\vv\in\sS^{d-1}$, there exists a constant $\nu_G(\vv)\in[0,+\infty[$ such that for any non-degenerate hyperrectangle $A$ normal to $\vv$, for any function $h:\,\sN\rightarrow \sR^+$ such that $\lim_{n\rightarrow \infty}h(n)=+\infty$, we have
 $$\lim_{n\rightarrow \infty}\frac{\tau_G(nA,h(n))}{\cH^{d-1}(nA)}=\nu_G(\vv)\text{ in $L^1$}.$$
 If moreover the origin of the graph belongs to $A$, or if  $\int_{\sR^+}x^{1+1/(d-1)}dG(x)<\infty$, then
  $$\lim_{n\rightarrow \infty}\frac{\tau_G(nA,h(n))}{\cH^{d-1}(nA)}=\nu_G(\vv)\text{  a.s.}.$$
  If the cylinder is flat, \textit{i.e.}, if $\lim_{n\rightarrow\infty} h(n)/n=0$, then the same convergence also holds for $\phi_G(nA,h(n))$.
  Either $\nu_G(\vv)$ is null for all $\vv \in\sS^{d-1}$ or $\nu_G(\vv)>0$ for all $\vv\in\sS^{d-1}$.
\end{thm}

In \cite{Zhang}, Zhang found a necessary and sufficient condition on $G$ under which $\nu_G(\vv)$ is positive. He proved the following result.
\begin{thm}
Let $G$ be a probability measure on $[0,+\infty[$ such that $\int_{\sR^+}xdG(x)<\infty$. Then, $\nu_G(\vv)>0$ if and only if $G(\{0\})<1-p_c(d)$.
\end{thm}
Let us give an intuition of this result. If $\tau_G(nA,h(n))>0$, then there exists a path in $\cyl'(nA,h(n))$ from the upper to the lower half part of its boundary such that all its edges have positive capacity. Indeed, if there does not exist such a path, there exists a cutset of null capacity and it contradicts $\tau_G(nA,h(n))>0$. Thus, the fact that $\nu_G(\vv)>0$ is linked with the fact that the edges of positive capacity percolate, \textit{i.e.}, $G(\{0\})<1-p_c(d)$. The main difficult part of this result is to study the critical case, \textit{i.e.}, $G(\{0\})=1-p_c(d)$.

In \cite{flowconstant}, Rossignol and Théret extended the previous results without any moment condition on $G$, they even allow $G$ to have an atom in $+\infty$ as long as $G(\{+\infty\})<p_c(d)$.  They proved the following law of large numbers for the maximal flow from the top to the bottom of flat cylinders.

\begin{thm}\label{RosTer}
For any probability measure $G$ on $[0,+\infty]$ such that $G(\{+\infty\})<p_c(d)$, for any $\vv\in\sS^{d-1}$, there exists a constant $\nu_G(\vv)\in [0,+\infty[$ such that for any non-degenerate hyperrectangle $A$ normal to $\vv$, for any function $h$ such that $h(n)/ \log n \rightarrow \infty$ and $h(n)/ n \rightarrow 0$ when $n$ goes to infinity, we have
$$\lim_{n\rightarrow \infty}\frac{\phi_G(nA,h(n))}{\cH^{d-1}(nA)}=\nu_G(\vv)\text{  a.s.}.$$
Moreover, for every $\vv\in\sS^{d-1}$,
$$\nu_G(\vv)>0 \iff G(\{0\})<1-p_c(d) \,.$$
\end{thm}

\begin{rk}
Note that if $G(\{0\})>1-p_c(d)$, then $G(\{+\infty\})<p_c(d)$ and the flow constant is well defined according to Theorem \ref{RosTer}.
\end{rk}

In \cite{Kesten:flows}, Kesten proved a result similar to Proposition \ref{prop1} for the rescaled maximal flow $\Phi_G(nA,h(n))/\cH^{d-1}(nA)$ in a straight cylinder. He worked in dimension $3$ and considered the more general case where the lengths of the sides of the cylinder go to infinity but at different speeds in every direction, under the technical assumption that $G(\{0\})$ is smaller than some small constant. He worked with dual sets, and he had to define properly the notion of surface. He had to deal with the fact that the flow $\Phi_G$ is not subadditive. His work was very technical and cannot be easily adapted to tilted cylinders because the arguments crucially depend on some symmetries of the model for straight cylinders. Zhang extended Kesten's result in higher dimensions and without any hypothesis on $G(\{0\})$ in \cite{Zhang2017}. The asymptotic behavior of maximal flows $\Phi_G(nA,h(n))$ through tilted and non-flat cylinders was studied by Cerf and Théret in \cite{CT4,CT3,CT2,CT1}. In those papers, they even considered maximal flow through more general domains than cylinders. 

The results we have gathered here concerning maximal flows are the analogues of known results for the time constant in the study of geodesics in first passage percolation (see for instance Kesten's lecture note \cite{Kesten:StFlour}). We summarize here a few of them. In this paragraph, we interpret the random variable $t_G(e)$ as the time needed to cross the edge $e$. The passage time $T_G(\gamma)$ of a path $\gamma$ corresponds to the time needed to cross all its edges, \textit{i.e.}, $T_G(\gamma)=\sum_{e\in\gamma}t_G(e)$, and a geodesic between two points $x$ and $y$ of $\sZ^d$ is a path that achieves the following infimum:
$$T_G(x,y)=\inf\{T_G(\gamma):\,\gamma\text{ is a path from $x$ to }y\}\,.$$
As the times needed to cross the edges are random, a geodesic is a random path. Under some moment conditions, for all $x\in\sZ^d$, $T_G(0,nx)/n$ converges a.s. to a time constant $\mu_G(x)$. The time constant $\mu_G$ is either identically null or can be extended by homogeneity and continuity into a norm on $\sR^d$. Kesten investigated the positivity of $\mu_G$ and obtained that $\mu_G>0$ if and only if $G(\{0\})<p_c(d)$, see Theorem 1.15 in \cite{Kesten:StFlour}. Intuitively if the edges of null passage time percolate, there exists an infinite cluster $\sC$ made of edges of null passage time. A geodesic from $0$ to $nx$ tries to reach the infinite cluster $\sC$ as fast as possible, then travels in the cluster $\sC$ at infinite speed and exits the cluster at the last moment to go to $nx$. Under some good moment assumptions, the time needed to go from $0$ to $\sC$ and from $\sC$ to $nx$ is negligible compared to $n$ . We can show in this case that $\mu_G(x)=0$.

\subsection{Background on the minimal length of a geodesic and the minimal size of a minimal cutset}

By the max-flow min-cut theorem, we know that $\Phi_G(A,h)$ is equal to the minimal capacity of cutsets that cut the top from the bottom of $\cyl(A,h)$. Among all the cutsets of minimal capacity we are interested in the ones with the minimal cardinality:
$$\psi_G(A,h,\vv):=\inf\left\{\card_e(E) : \begin{array}{c}\text{ $E$ cuts the top from the bottom of}\\\text{$\cyl(A,h)$ and $E$ has capacity $\Phi_G(A,h)$}\end{array}\right\}$$
where $\card_e(E)$ denotes the number of edges in the edge set $E$.

The study of the quantity $\psi_G(A,h,\vv)$ was initiated by Kesten in \cite{Kesten:flows} in dimension $3$ for straight boxes and distributions $G$ such that $G(\{0\})<p_0$ where $p_0$ has to be small enough. Let $k,\,l,\,,m\in\sN$, we define the straight box $B(k,l,m)=[0,k]\times[0,l]\times[0,k]$.
\begin{thm} Let $k,\,l \, ,m\in\sN$. There exists a $p_0>1/27$ such that for all distributions $G$ on $[0,+\infty)$ such that $G(\{0\})<p_0$, there exist constants $\theta$, $C_1$ and $C_2$ depending on $G$ such that for all $n\geq 0$,
$$\Prb\left[\begin{array}{c}\text{there exists a dual set $E^*$ of at least $n$ plaquettes that cuts}\\\text{the top from the bottom of the box $B(k,l,m)$, which}\\\text{contains the point $(-\frac{1}{2},-\frac{1}{2},\frac{1}{2})$ and  such that $T_G(E^*)\leq \theta n$}\end{array}\right]\leq C_1 e^{-C_2 n}\,.$$

\end{thm}

 Zhang in \cite{Zhang2017} extended this result in all dimensions and for distributions $G$ such that $G(\{0\})<1-p_c(d)$ and with an exponential moment. He obtained the following result.
\begin{thm} Let $G$ be a distribution on $\sR^+$ such that $\int_0^\infty \exp(\eta x)dG(x)<\infty$ for some $\eta>0$ and $G(\{0\})<1-p_c(d)$. Let $k_1,\dots,k_{d-1}\in\sN$ and $h$ with $\log h \leq k_1\cdots k_{d-1}$. Let $A=\prod_{i=1}^{d-1}[0,k_i]\times \{0\}$. There exist constants $\beta\geq 1$ depending on $G$ and $d$, $C_1$ and $C_2$ depending on $G$, $d$ and $\beta$  such that for all $\lambda >\beta$,
$$\Prb[\psi_G(A,h,\vv)>\lambda \cH^{d-1}(A)]\leq C_1\exp(-C_2  \cH^{d-1}(A))\,.$$
\end{thm}
He proves this result by a geometrical construction of a smooth cutset at a mesoscopic level, by combinatorial considerations and by percolation estimates. His work can be extended to tilted cylinders. However, his proof relies crucially on the fact that $G(\{0\})<1-p_c(d)$ and cannot be adapted for $G(\{0\})>1-p_c(d)$.

The aim of this article is to understand the behavior of $\psi_G(A,h,\vv)$ in the supercritical case $G(\{0\})>1-p_c(d)$ (the critical case $G(\{0\})=1-p_c(d)$ is expected to be much more delicate to study).
This is the analogous problem in higher dimension of the study of the minimal length of a geodesic. We denote by $N_G(x,y)$ the minimal length of a geodesic between $x$ and $y$:
$$N_G(x,y)=\inf\{|\gamma|\,:\,\gamma \text{ is a geodesic between $x$ and $y$}\}\,.$$
The analog question is to study how does $N_G(0,nx)$, the minimal length of a geodesic between $0$ and $nx$,  grow when $n$ goes to infinity if $G(\{0\})> p_c(d)$. It is expected to grow at speed $n$. This result was first proved by Zhang and Zhang in dimension 2 in \cite{ZhangZhang}. 
 
 \begin{thm} Let $d=2$ and let $G$ be a distribution on $[0,+\infty[$ such that $G(\{0\})>1/2$. We have 
$$\lim_{n\rightarrow \infty}\frac{N_G((0,0),(0,n))}{n}=\lambda_{G(\{0\})}\text{  a.s. and in $L^ 1$}$$
 where $\lambda_{G(\{0\})}$ depends only on $G(\{0\})$. 
 \end{thm}
Zhang later extended this result to all dimensions under the condition that $G(\{0\})>p_c(d)$ in \cite{ZHANGsb}. 

\begin{rk}These works can be extended to all directions. To extend it to rational directions we can use a subadditive ergodic theorem and instead of considering the points $0$ and $nx$, it is more convenient to  consider their regularized version $\widetilde{0}$ and $\widetilde{nx}$, \textit{i.e.}, their projection on the infinite cluster of null passage time (see \cite{cerf2016}). We can show that $\lim_{n\rightarrow\infty} N_G(0,nx)/n = \lim_{n\rightarrow\infty} N_G(\widetilde{0},\widetilde{nx})/n $. By continuity, we can also extend it to irrational directions.
\end{rk}
\subsection{Main result and idea of the proof}
The main result of this paper is the following.
\begin{thm}\label{heart} Let $d\geq 3$. Let $G$ be a distribution on $[0,+\infty]$ such that $G(\{0\})>1-p_c(d)$. Let $\vv\in\sS^{d-1}$. There exists a finite constant $\zeta_{G(\{0\})}(\vv)$ such that for all function $h$ such that $h(n)/ \log n \rightarrow \infty$ and $h(n)/ n \rightarrow 0$ when $n$ goes to infinity, for all non-degenerate hyperrectangle $A$ normal to $\vv$, 
$$\lim_{n\rightarrow \infty}\frac{\psi_G(nA,h(n),\vv)}{\cH^{d-1}(nA)}=\zeta_{G(\{0\})}(\vv)\text{ a.s..}$$
The constant $\zeta_{G(\{0\})}(\vv)$ depends only the direction $\vv$, $G(\{0\})$ and $d$  and not on $A$ itself nor $h$. 
\end{thm}

In the following, if a function $h:\sN\rightarrow \sR_+$ satisfies $h(n)/ \log n \rightarrow \infty$ and $h(n)/ n \rightarrow 0$ when $n$ goes to infinity, we say that $h$ satisfies condition $(\star)$.

To prove Theorem \ref{heart}, we need to introduce an alternative flow in section \ref{s3} that is inspired from \cite{flowconstant}. There are two issues: we need to study cutsets that may be merged together into a cutset and that have null capacity.  Although the cutsets corresponding to the flow $\tau$ in adjacent cylinders may be glued together easily, these cutsets do not have null capacity in general: the union of two cutsets of minimal capacity is a cutset  but does not have minimal capacity. The flow $\tau$ is subadditive but not the minimal cardinal of the minimal corresponding cutsets. The alternative flow we build in section \ref{s3} is such that the maximal flow is always null and if we merge two adjacent cutsets for this flow it is still a cutset. The aim is to work only with cutsets of null capacity so when we merge two cutsets together the union has null capacity and is therefore of minimal capacity.

Let $\chi_G$ be the minimal cardinality of a minimal cutset for the alternative flow we will define in section \ref{s3}. 
First, we  show the convergence for the expected value of $\chi_G$, properly renormalized, by using subadditive arguments in section \ref{s4}. The proof enables us to say that the limit does not depend on $h$ nor on $A$. Next, we prove that the alternative flow we have defined is actually very similar to the flow through the cylinder. We prove in section \ref{s5} that the limit obtained in \ref{s4} is equal to the limit of the renormalized expected value of $\psi_G$. In section \ref{s6}, we use a concentration inequality on $\psi_G$ to show that this random variable is close to its expectation and thus we prove Theorem \ref{heart}.

\subsection{More notations and useful results}
 For any vertex set $\sC\subset \sZ^d$, we define its diameter by $\Diam(\sC)=\sup\{\|x-y\|_2\,:\, x,y \in \sC\}$, its cardinal $\card_v(\sC)$ by the number of vertices in $\sC$, and its exterior edge boundary $\partial_e \sC$ by
$$\partial_e\sC=\left\{\langle x,y \rangle \in \E^d\,:\,\begin{array}{c}x\in\sC,\, y\notin \sC \text{ and there exists} \\\text{a path from $y$ to infinity in $\sZ^d\setminus \sC$}
\end{array} \right\}\,.$$
The notation $\langle x,y \rangle$ corresponds to the edge of endpoints $x$ and $y$. We recall that for any edge set $E\subset \E^d$, $\card_e(E)$ denotes the number of edges in $E$. There exits a constant $c_d$ such that for any finite connected set $\sC$ of vertices, $\card_e(\partial_e\sC)\leq c_d\card_v(\sC)$. Note that when there is no ambiguity we will denote by $|E|$ the cardinal of the set $E$.
We define the exterior $\ext(E)$ of a set of edges $E$:
$$\ext(E)=\{x\in\sZ^d\,:\,\text{there exists a path from $x$ to infinity in } \E^d\setminus E\}\,.$$
We will also need the notion of $r$-neighborhood $\cV(H,r)$ of a subset $H$ of $\sR^d$ defined by 
$$\cV(H,r)=\{x\in\sR^d,\,d(x,H)<r\}$$
where the distance is the Euclidian one : $d(x,H)=\inf\{\|x-y\|_2,\,y\in H\}$ and $\|z\|_2=\sqrt{\sum_{i=1}^d z_i^2}$ for $z=(z_1,\dots,z_d)\in\sR^d$.

Let $x\in\sZ^d$, we denote by $\sC_{G,0}(x)$ the connected component of $x$ in the percolation $(\ind_{t_G(e)>0})_{e\in\E^d}$, which can be seen as an edge set and as a vertex set.

The following theorem is a classical result on percolation that enables us to control the probability that an open cluster $\sC_{G,0}(x)$ is big in the subcritical regime, \textit{i.e.}, when $\Prb[t_G(e)>0]<p_c(d)$ (see for instance Theorem (6.1) and (6,75) in \cite{grimmettt:percolation}).
\begin{thm}\label{controlcluster}
Let us assume $G(\{0\})>1-p_c(d)$. There exist two constants $\kappa_1$ and $\kappa_2$ depending only on $G(\{0\})$ such that for all $x\in\sZ^d,\,n\in\sN$,
\begin{align}
\Prb[\card_v(\sC_{G,0}(x))>n]\leq \kappa_1\exp(-\kappa_2 n)\,.
\end{align}
\end{thm}

We state here the following concentration result that will be useful in section \ref{s6} (see 2.5.4 in \cite{massart2007concentration}).
\begin{prop}[Efron's Stein inequality]\label{ESI}
Suppose that $X_1, X'_1,\dots,X_n,X'_n$ are independent random variables and that for all $i\in\{1,\dots,n\}$, $X_i$ and $X'_i$ have the same distribution. We denote by $X$ the vector $(X_1,\dots,X_n)$ and by $X^{(i)}$ the vector $(X_1,\dots,X'_i,\dots,X_n)$. 

For all functions $f:\sR ^n \rightarrow \sR$,
$$\Var(f(X))\leq \dfrac{1}{2}\sum_{i=1}^n \E\left[(f(X)-f(X^{(i)}))^2\right]=\sum_{i=1}^n \E\left[(f(X)-f(X^{(i)}))_-^2\right]$$
where for any real $t$, $t_-=\max(0,-t)$.
\end{prop}

\section{Definition of an alternative flow}\label{s3}
 Instead of directly considering a smallest minimal cutset for the cylinder, we are going to study a different object which is more convenient for our purpose. 

Let $\overrightarrow{v}\in\sS^{d-1}$, and let $A$ be any non-degenerate hyperrectangle normal to $\overrightarrow{v}$. We denote by $\hyp(A)$ the hyperplane spanned by $A$ defined by 
$$\hyp(A)=\{x+\overrightarrow{w}\,:\, x\in A,\,\overrightarrow{w}\cdot \vv=0\}$$
where $\cdot$ denotes the usual scalar product on $\sR^d$.
For any $h>0$, we denote by $\slab(A,h,\overrightarrow{v})$ (resp. $\slab(A,\infty,\vv)$) the slab of basis the hyperplane spanned by $A$ and of height $h$ (resp. of infinite height), \textit{i.e.}, the subset of $\sR^d$ defined by 
$$\slab(A,h,\overrightarrow{v})=\{x+r\overrightarrow{v} : x\in\hyp(A),r\in[0,h]\} $$
(resp. $\slab(A,\infty,\overrightarrow{v})=\{x+r\overrightarrow{v} : x\in\hyp(A),\, r\geq 0\}\, $).
We are going to consider a thicker version of $A$, namely $\cyl(A,d)$, that we will denote by $\bar{A}$ for short.
Let $W(A,h,\vv)$ be the following set of vertices in $\sZ^d$, which is a discretized version of $\hyp (A+h\vv)$: 
$$W(A,h,\vv):=\left\{\begin{array}{c}x\in\sZ^d\cap \slab(A,h,\overrightarrow{v}):\\
\exists y \in \sZ^d\cap (\slab(A,\infty,\overrightarrow{v})\setminus \slab(A,h,\overrightarrow{v})),\langle x,y \rangle \in \E^d
\end{array} \right\}\, .$$

We say that a path $\gamma=(x_0,e_1,x_1,\dots,e_n,x_n)$ goes from $\bar{A}$ to $\hyp (A+h\vv)$ in $\slab(A,h,\overrightarrow{v})$ if :
\begin{itemize}
\item $\forall i \in \{0,\dots,n\}, x_i\in \slab(A,h,\overrightarrow{v})$
\item $x_0\in \bar{A}$
\item $x_n \in W(A,h,\vv)$.
\end{itemize}

 \begin{figure}[H]
\def\svgwidth{1.1\textwidth}
\begin{center}
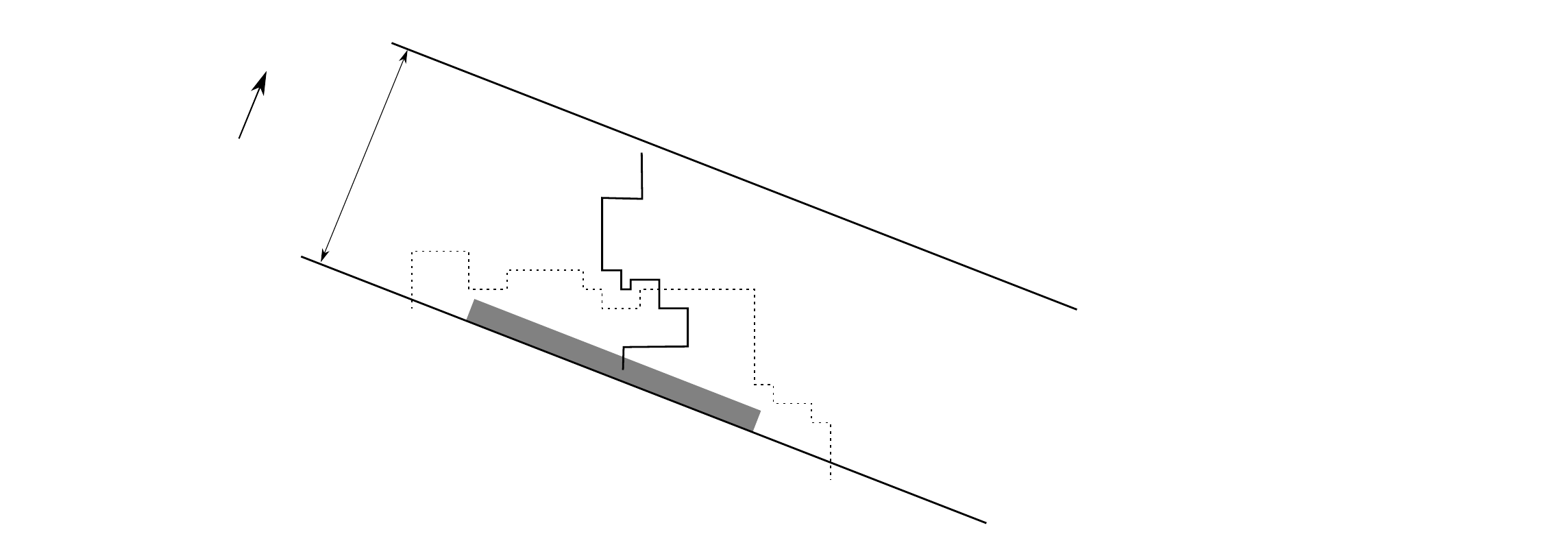
\caption[fig1]{\label{fig1}Dual of a set of edges that cuts $\overline{A}$ from $\hyp(A+h\vv)$ in $\slab(A,h,\vv)$.}
\end{center}
\end{figure}
We say that a set of edges $E$ cuts $\bar{A}$ from $\hyp (A+h\vv)$ in  $\slab(A,h,\overrightarrow{v})$ if $E$ contains at least one edge of any path $\gamma$ that goes from $\bar{A}$ to $\hyp(A+h\vv)$ in  $\slab(A,h,\overrightarrow{v})$, see Figure \ref{fig1}.

If all the clusters $\sC_{G,0}(x)$ for $x\in \bar{A}$ have a diameter less than $h/2$, then there exists a set of edges that cuts $\bar{A}$ from $\hyp(A+h\vv)$ in $\slab(A,h,\overrightarrow{v})$ of capacity $0$ (take for instance the intersection of the set $\bigcup_{x\in \bar{A}\cap\sZ^d}\partial_e  \sC_{G,0}(x)$ with $\slab(A,h,\vv)$). Working with cutsets of null capacity is interesting because the union of two cutsets of null capacity is of null capacity and therefore achieve the minimal capacity among all cutsets. This is not the case if one of them has positive capacity. Thus instead of considering a deterministic $h$, we are going to consider a random height $H_{G,h}(A)$ as 
$$H_{G,h}(A)=\inf\left\{t\geq h : \left(\bigcup_{x\in\cyl(A,h/2)\cap \sZ^d}\sC_{G,0}(x)\right)\cap W(A,t,\vv)=\emptyset\right\} \,.$$
The definition of $H_{G,h}(A)$ ensures the existence of a null cutset between $\bar{A}$ and $\hyp (A+H_{G,h}(A)\vv)$ for $h\geq 2d$.
\begin{lem}\label{lem1}Let $G$ be a distribution on $[0,+\infty]$ such that $G(\{0\})>1-p_c(d)$. Let $\vv \in \sS^{d-1}$. Let $A$ be a non-degenerate hyperrectangle normal to $\vv$ and $h>2d$ a positive real number.
The set $E=\bigcup_{x\in \bar{A}\cap \sZ^d} \partial_e\sC_{G,0}(x)$ cuts $\bar{A}$ from $\hyp (A+H_{G,h}(A)\vv)$ in $\slab(A,H_{G,h}(A),\vv)$ and has null capacity.
\end{lem}

\begin{proof}
Let $\vv \in \sS^{d-1}$, let $A$ be a non-degenerate hyperrectangle and $h>2d$.
Let $\gamma$ be a path from $x\in \bar{A}$ to $y\in W(A,H_{G,h}(A),\vv)$ in $\slab(A,H_{G,h}(A),\vv)$. By definition of $H_{G,h}(A)$, we have $W(A,H_{G,h}(A),\vv)\cap \left(\cup_{z\in \bar{A}} \sC_{G,0}(z)\right)=\emptyset$, thus  $y\in\ext(\partial_e\sC_{G,0}(x))$ and $\gamma$ must contain an edge in $\partial_e\sC_{G,0}(x)$. We conclude that $E$ is indeed a cutset between $\bar{A}$ and $\hyp (A+H_{G,h}(A)\vv)$ in $\slab(A,H_{G,h}(A),\vv)$. As all edges in the exterior edge boundary of a $\sC_{G,0}(x)$ have null capacity, the set $E$ is a cutset of null capacity.
\end{proof}

\begin{rk}
This definition of $H_{G,h}(A)$ may seem complicated, but the idea behind is simple. The aim was initially to build a random height $H_{G,h}(A)$ such that the minimal cutset between $A$ and $\hyp (A+H_{G,h}(A)\vv)$ has null capacity. This idea finds its inspiration from the construction of the subadditive object in section 4 in \cite{flowconstant}. However, because of technical issues that appear in the section \ref{s4}, we could not choose $H_{G,h}(A)$ as the smallest height such that there exists a cutset of null capacity between $A$ and $\hyp (A+H_{G,h}(A)\vv)$. The definition of $H_{G,h}(A)$ needs to also depend on the finite clusters $\sC_{G,0}(z)$ of $z\in\cyl(A,h/2)$. 
\end{rk}

For the rest of this section, we will work with cutsets of null capacity and we do not need to check if cutsets have minimal capacity. Among all the cutsets that achieve the minimum capacity, we are interested in the ones with the smallest size. We denote by $\chi_G(A,h,\vv)$ the following quantity :
\begin{align} \label{defchi}
\chi_G(A,h,\vv) = \inf \left\{\card_e(E) : \begin{array}{c}
\text{ $E$ cuts $\bar{A}$ from  $\hyp (A+H_{G,h}(A)\vv)$ in} \\\text{$\slab(A,H_{G,h}(A),\vv)$ and $T_G(E)=0$} \end{array} \right\}\,.
\end{align}
\begin{rk}
Because of another technical difficulty that appears in section \ref{s4} we choose to make appear $\bar{A}$ instead of $A$ in the definition of $\chi_G(A,h,\vv)$. We need the cutset not to be to close from $A$ in the proof of Proposition \ref{subb}, taking $\bar{A}$ instead of $A$ prevents this situation from happening.
\end{rk}
As a corollary of Lemma \ref{lem1}, we know that $\chi_G(A,h,\vv)$ is finite and we have the following control
\begin{align}\label{control}
\chi_G(A,h,\vv)&\leq \card_e(E)\nonumber\\
&\leq\sum_{x\in \bar{A}\cap \sZ^d} \card_e(\partial_e\sC_{G,0}(x))\nonumber\\
&\leq\sum_{x\in \bar{A}\cap \sZ^d} c_d\card_v(\sC_{G,0}(x))\,.
\end{align}
Thanks to Theorem \ref{controlcluster}, as $G(\{0\})>1-p_c(d)$, almost surely for all $x\in\sZ^d$, the cluster $\sC_{G,0}(x)$ is finite thus $\chi_G(A,h,\vv)\leq \sum_{x\in \bar{A}\cap \sZ^d} c_d\card_v(\sC_{G,0}(x))<\infty$ a.s..

We expect $\chi_G(A,h,\vv)$ to grow at order $\cH^{d-1}(A)$ when the side lengths of $A$ go to infinity. We want first to show that $\lim_{n\rightarrow \infty}\E[\chi_G(nA,h(n),\vv)]/\cH^{d-1}(nA)$ exists, is finite and does not depend on $A$ but only on $\vv$ and $G(\{0\})$.

\section{Subadditive argument}\label{s4}

In this section, we show the convergence of
$\E[\chi_G(nA,h(n),\vv)]/\cH^{d-1}(nA)$, see Proposition \ref{subb} below. This proof relies on subadditive arguments. However, we do not use a subadditive ergodic theorem for two reasons: we want to study this convergence for all directions (included irrational ones) and all hyperrectangles, and we aim to show that the limit does not depend on the hyperrectangle $A$ nor on the height function $h$.
 
\begin{prop}\label{subb} Let $G$ be a distribution on $[0,+\infty]$ such that $G(\{0\})>1-p_c(d)$. For every function $h$ satisfying condition $(\star)$, for every $\vv\in\sS^{d-1}$, for every non-degenerate hyperrectangle $A$ normal to $\vv$, the limit 
$$\zeta_{G(\{0\})}(\vv):=\lim_{n\rightarrow \infty}\frac{\E[\chi_G(nA,h(n),\vv)]}{\cH^{d-1}(nA)}$$
exists and is finite. It depends only the direction $\vv$, on $G(\{0\})$ and on $d$ but not on $A$ itself nor $h$. 
\end{prop}
The proof of this proposition is inspired by the proof of Proposition 3.5. in \cite{Rossignol2010}. This idea was already present in \cite{Cerf:StFlour}. 
In fact we mimic the beginning (\textit{i.e.}, the easy part) of the proof of the subadditive ergodic theorem.

\begin{proof}
Let $\vv\in\sS^{d-1}$. Let us consider two non-degenerate hyperrectangles $A$ and $A'$ which are both orthogonal to the unit vector $\vv$, and two height functions $h,h' : \sN\rightarrow \sR^{+}$ that respect condition $(\star)$. As $\lim_{n\rightarrow \infty} h(n)=\lim_{n\rightarrow \infty} h'(n)=\infty$, if we take $n\in \sN$, there exists an $N_0(n)$ such that for all $N\geq N_0(n)$, we have $h(N)\geq h'(n)+2d+1$ and $N\Diam (A) > n\Diam (A')$ (where $\Diam(A)=\sup\{\|x-y\|_2\,:\, x,y \in A\}$). Our goal is to cover the biggest hyperrectangle $NA$ by translates of $nA'$. We do not want to cover the whole hyperrectangle $NA$ but at least the following subset of $NA$:
$$D(n,N):=\{x\in NA\, | \,d(x,\partial(NA))>2n\Diam A'\}\,,$$
where $\partial(NA)$ denotes the boundary of $NA$.
 \begin{figure}[h]
\def\svgwidth{0.8\textwidth}
\begin{center}
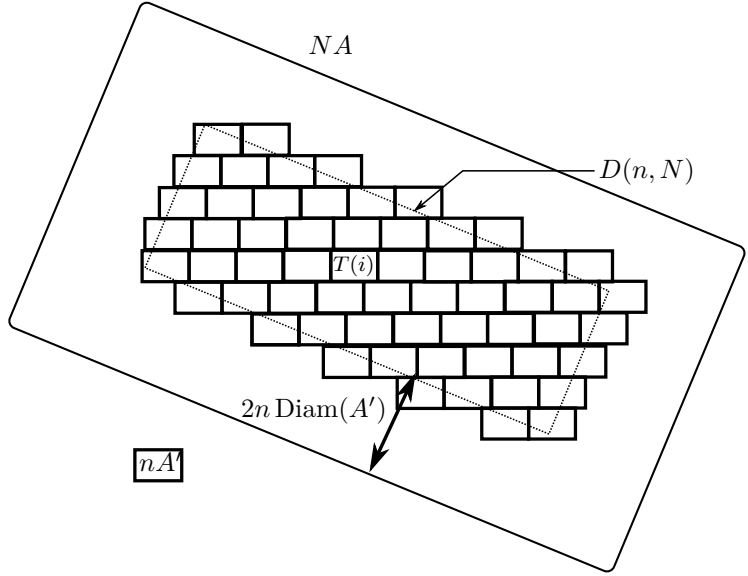
\caption{\label{fig2}Decomposition of $NA$ in translates of $nA'$}
\end{center}
\end{figure}

There exists a finite collection of hyperrectangles $(T(i))_{i\in I}$ such that $T(i)$ is a translate of $nA'$, each $T(i)$ intersects $D(n,N)$, the collection $(T(i))_{i\in I}$ have pairwise disjoint interiors, and their union $\cup_{i\in I}T(i)$ contains the set $D(n,N)$ (see Figure \ref{fig2}). By definition of $D(n,N)$, we also have that the union $\cup_{i\in I}T(i)$ is contained in $NA$.

The quantities $\E[\chi_G(T'(i),h'(n),\vv)]$ and $\E[\chi_G(nA',h'(n),\vv)]$ are not necessarily equal. Indeed, $T(i)$ is the translate of $nA'$ by a non-integer vector in general. Thus, instead of considering $T(i)$, let us consider $T'(i)$ which is the image of $nA'$ by an integer translation, and $T'(i)$ is the translated of $T(i)$ by a small vector. We want to choose $T'(i)$ such that $T'(i)\subset\slab(NA,h,\vv)$. More precisely, for all $i\in I$, there exist two vectors $\overrightarrow{t_i} \in \sR^{d}$ and $\overrightarrow{t_i}' \in \sZ^{d}$ such that $\|\overrightarrow{t_i}\|_\infty <1$, $\overrightarrow{t_i}\cdot \vv\geq 0$, $T'(i)=T(i)+\overrightarrow{t_i}$ and $T'(i)=nA'+\overrightarrow{t_i}'$. As for all $i\in I$, $\overrightarrow{t_i}\cdot \vv<\sqrt{d}$, the union $\cup_{i\in I}T'(i)$ is contained in $\slab(NA,d,\vv)$ (see Figure \ref{fig6}).

As for all $i\in I$, $T'(i)\in\slab(NA,d,\vv)$ and $h'(n)+2d<h(N)$, then we have $\cyl(T'(i),h'(n)/2)
\subset 
\cyl(NA,h(N)/2)$, and by definition of the random height  $\slab(T'(i),H_{G,h'(n)}(T'(i)),\vv)\subset\slab(NA,H_{G,h(N)}(NA),\vv)$.


The family $(\chi_G(T'(i),h(n),\vv))_{i\in I}$ is identically distributed but not independent. For all $i\in I$, let $E_i$ be a set that satisfies the infimum in the definition of $\chi_G(T'(i),h(n),\vv)$. We want to build from the family $(E_i)_{i\in I}$ a set of null capacity that cuts $\overline{NA}$ from $\hyp(NA+H_{G,h(N)}(NA))$ in $\slab(NA,H_{G,h(N)}(NA),\vv)$ on the event 
\begin{align}\label{event:F}
\cF_{n,N}=\bigcap _{x\in \cyl(NA,h(n)/2+d)}\left\{\card_v(\sC_{G,0}(x))<\frac{h(N)}{4}\right\}\,.
\end{align}

 \begin{figure}[H]
\def\svgwidth{1\textwidth}
\begin{center}
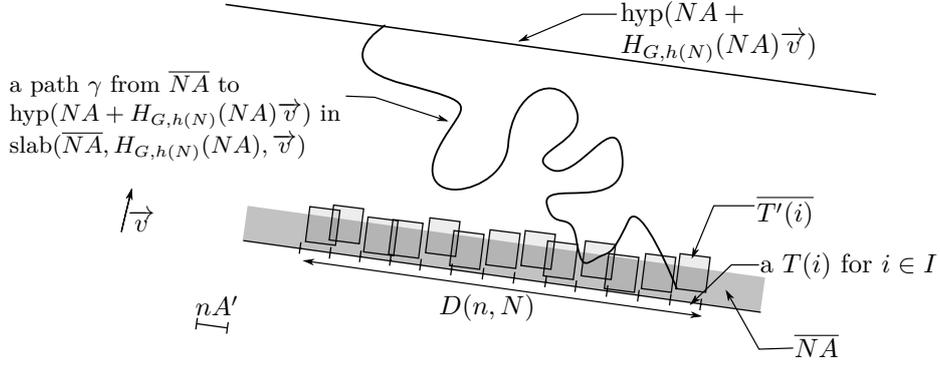
\caption{\label{fig6}Representation of the $\overline{T'(i)}$ for $i\in I$.}
\end{center}
\end{figure}

We fix $r=4d$. Let $V^{1}_{0}$ (resp. $V_0^{2}$, $V_0^{3}$, $V_0$) be the set of vertices included in $\cE^{1}_{0}$ (resp. $\cE_0^{2}$, $\cE_0^{3}$, $\cE_0$), where we define 

$$\cE^{1}_{0}=\bigcup_{i\in I}\cV(\partial T'(i),r),$$
$$\cE^{2}_{0}=\cV(NA\setminus D(n,N),r),$$
$$\cE^{3}_{0}=\cV(\cyl(\partial(NA),h(N)/2),r)$$
and
$$\cE_0=\cE^{1}_{0}\cup \cE^{2}_{0}\cup \cE^{3}_{0}\,.$$

Let us show that the set $E=\left( \cup_{i\in I} E_i\right) \cup \left( \cup_{x\in V_0}\partial_e\sC_{G,0}(x)\right)$ cuts $\overline{NA}$ from $\hyp(NA+H_{G,h(N)}(NA)\vv)$ in $\slab(NA,H_{G,h(N)}(NA),\vv)$ on the event $\cF_{n,N}$. Let $\gamma$ be a path from $x\in \overline{NA}$ to $y\in W(NA,H_{G,h(N)}(NA),\vv)$ that stays in $\slab(NA,H_{G,h(N)}(NA),\vv)$, we denote it by $\gamma=(x=v_0,e_1,v_1,\dots,e_m,v_m=y)$. Let us consider the last moment when $\gamma$ exits $\overline{NA}$, \textit{i.e.}, we define $p=\inf\{i\in\{0,\dots,m\},\forall j>i,\,v_j\notin \overline{NA}\}$. We distinguish several cases.
\begin {itemize}
\item If the edge $e_{p+1}$ cuts $\cyl(\partial(NA),d)\cup (NA\setminus D(n,N)+d\vv)$, then $v_p\in \cE^2_0$ and by definition of $H_{G,h(N)}(NA)$, we have $v_m=y\notin \sC_{G,0}(v_p)$, so $\gamma \cap \partial_e\sC_{G,0}(v_p)\neq \emptyset$.
\item If the edge  $e_{p+1}$ cuts $(D(n,N)+d\vv)\setminus (\bigcup_{i\in I} \overline{T'(i)})$, we define $\pi$ the orthogonal projection on $\hyp(NA)$ and $z=e_{p+1}\cap \hyp(NA+d\vv)$. As $\pi(z)\in D(n,N)$, there exists an $i\in I$ such that $\pi(z)\in T(i)$, $\pi(z)\notin \pi(T'(i))$  and so $\pi(z)\in T(i)\setminus \pi(T'(i))$. Moreover, as $T'(i)=T(i)+\overrightarrow{t_i}$ where $\|\overrightarrow{t_i}\|_\infty<1$, we get that $\pi(z)\in \cV(\pi(\partial T'(i)),d)\cap \hyp(NA)$ and $v_p\in \cV(\partial T'(i),r)\subset \cE^1 _0$. We conclude as in the previous case that $\gamma \cap \partial_e\sC_{G,0}(v_p)\neq \emptyset$.
\item If there exists an $i\in I$ such that the edge $e_{p+1}$ cuts $\overline{T'(i)}\cap \hyp(NA+d\vv)$ we distinguish two cases. If $v_p\notin \overline{T'(i)}$, as $T'(i)\subset \slab(NA, \sqrt{d},\vv)$, the vertex $v_p$ cannot be "under" $T'(i)$, \textit{i.e.}, in $\slab(NA,\overrightarrow{t_i}\cdot \vv ,\vv)$. Therefore, the vertex $v_p$ belongs to $\cV(\partial T'(i),d)\subset \cE_0^1$, we conclude as in the previous cases that $\gamma\cap \partial _ e\sC_{G,0}(v_p)$. We now assume that $v_p\in \overline{T'(i)}$ (see Figure \ref{fig7}). Let us consider the first time after $p$ when $\gamma$ cuts $\hyp(T'(i)+H_{G,h'(n)}(T'(i))\vv)\cup \cyl(\partial(NA),h(N)/2)$. On the event $\cF_{n,N}$, we have the following events: $\slab(T'(i),H_{G,h'(n)}(T'(i)),\vv)\subset \slab(NA,h(N)/2,\vv)$, $v_{p+1}\in \slab(T'(i),H_{G,h'(n)}(T'(i)),\vv)\cap \cyl(NA,h(N)/2)$ and the vertex $y$ does not belong to $ \slab(T'(i),H_{G,h'(n)}(T'(i)),\vv)\cap \cyl(NA,h(N)/2)$. Therefore, $l=\inf\{j>p+1,\, e_j \text{ cuts } \hyp(T'(i)+H_{G,h'(n)}(T'(i))\vv)\cup \cyl(\partial(NA),h(N)/2)\}$ is well defined. Note that by definition of $v_p$, $\gamma$ cannot exit $\slab(T'(i),H_{G,h'(n)}(T'(i)),\vv)$ by $\hyp(T'(i))$, otherwise, $\gamma$ would come back in $\overline{NA}$.
Next, if $e_l$ cuts $\cyl(\partial(NA),h(N)/2)$ then $v_{l-1}\in \cE_0^3$ and by definition of $H_{G,h(N)}(NA)$, $y\notin \sC_{G,0}(v_{l-1})$ and $\gamma \cap \partial_e \sC_{G,0}(v_{l-1})\neq \emptyset$. 
Let us now assume that $e_l$ cuts $\hyp(T'(i)+H_{G,h'(n)}(T'(i))\vv)$, then $v_{l-1}\in W(T'(i),H_{G,h'(n)}(T'(i)),\vv)$, the portion of $\gamma$ from $v_p$ to $v_{l-1}$ is a path from $\overline{T'(i)}$ to  $\hyp(T'(i)+H_{G,h'(n)}(T'(i))\vv)$ that, by definition of $v_{l-1}$, stays in $\slab(T'(i),H_{G,h'(n)}(T'(i)),\vv)$ and so $\gamma\cap E_i\neq \emptyset$.

\end{itemize}

 \begin{figure}[H]
\def\svgwidth{1\textwidth}
\begin{center}
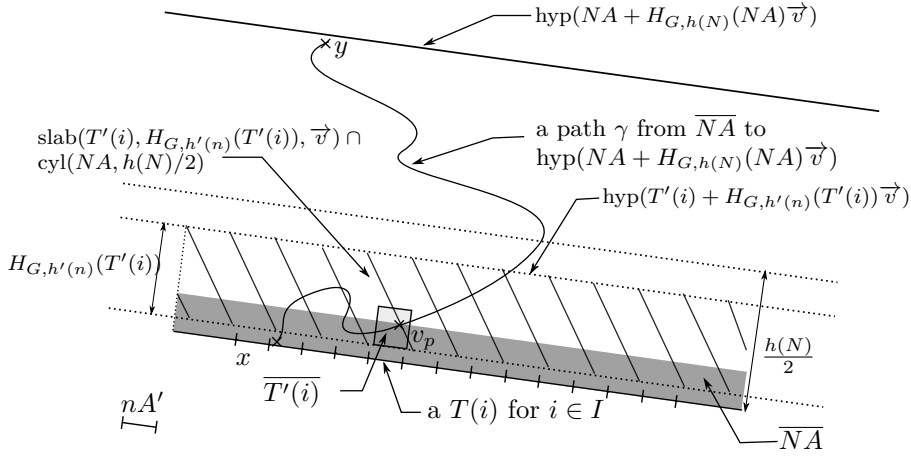
\caption[fig7]{\label{fig7}A path from $\overline{NA}$ to $\hyp(NA+H_{G,h(N)}(NA)\vv)$ in $\slab(NA,H_{G,h(N)}(NA),\vv)$ such that $v_p \in \overline{T'(i)}$ for an $i \in I$ .}
\end{center}
\end{figure}
Therefore, we conclude that the set $E$ cuts $\overline{NA}$ from $\hyp(NA+H_{G,h(N)}(NA))$ in $\slab(NA,H_{G,h(N)}(NA),\vv)$ on the event $\cF_{n,N}$. Since for all $i\in I$, the set $E_i$ has null capacity and for any $x\in \sZ^{d}$, the set $\partial_e\sC_{G,0}(x)$ contains only edges with null capacity, the set $E$ itself has null capacity. Thus, we can upperbound the quantity $\chi_G(NA,h(N),\vv)$ by the size of $E$ on the event $\cF_{n,N}$ and by the size of $\bigcup_{x\in \overline{NA}\cap \sZ^d} \partial_e\sC_{G,0}(x)$ on the event $\cF_{n,N}^c$ (by Lemma \ref{lem1}):
\begin{align*}
\chi_G(NA,h(N),\vv)&\leq \chi_G(NA,h(N),\vv)\ind_{\cF_{n,N}}+\chi_G(NA,h(N),\vv)\ind_{\cF_{n,N}^c} \\
&\leq\sum_{i\in I}|E_i|+\sum_{x\in V_0}|\partial_e\sC_{G,0}(x)|+\left(\sum_{x\in \overline{NA}\cap \sZ^d}|\partial_e\sC_{G,0}(x)|\right) \ind_{\cF_{n,N}^c}\\
&\leq \sum_{i\in I}\chi_G(T'(i),h'(n),\vv)+\sum_{x\in V_0}|\partial_e\sC_{G,0}(x)|\\
&\hspace{1cm}+\left(\sum_{x\in \overline{NA}\cap \sZ^d}|\partial_e\sC_{G,0}(x)|\right) \ind_{\cF_{n,N}^c}\, .
\end{align*}
Taking the expectation we get 
\begin{align}\label{expchi}
\frac{\E[\chi_G(NA,h(N),\vv)]}{\cH^{d-1}(NA)}&\leq \sum_{i\in I}\frac{\E[\chi_G(T'(i),h'(n),\vv)]}{\cH^{d-1}(NA)}+\sum_{x\in V_0}\frac{\E[|\partial_e\sC_{G,0}(x)|]}{\cH^{d-1}(NA)}\nonumber\\
&\hspace{1cm}+\sum_{x\in \overline{NA}\cap \sZ^d}\frac{\E[|\partial_e\sC_{G,0}(x)|\ind_{\cF_{n,N}^c}]}{\cH^{d-1}(NA)}\nonumber\\
&\leq |I|\frac{\E[\chi_G(nA',h'(n),\vv)]}{\cH^{d-1}(NA)}+\frac{\card_v(V_0)}{\cH^{d-1}(NA)}\E[|\partial_e\sC_{G,0}(0)|]\nonumber\\
&\hspace{1cm}+\sum_{x\in \overline{NA}\cap \sZ^d}\frac{\sqrt{\E[|\partial_e\sC_{G,0}(x)|^2] \Prb[\cF_{n,N}^c]}}{\cH^{d-1}(NA)}\nonumber\\
&\leq \frac{\E[\chi_G(nA',h'(n),\vv)]}{\cH^{d-1}(nA')}+\frac{\card_v(V_0)}{\cH^{d-1}(NA)}\E[|\partial_e\sC_{G,0}(0)|]\nonumber\\
&\hspace{1cm}+\card(\overline{NA}\cap \sZ^d)\frac{\sqrt{\E[|\partial_e\sC_{G,0}(0)|^2] \Prb[\cF_{n,N}^c]}}{\cH^{d-1}(NA)}
\end{align}
where we use in the second inequality Cauchy-Schwartz' inequality. 
By definition of $\cF_{n,N}$ (see \eqref{event:F}) and using Theorem \ref{controlcluster}, we obtain the following upperbound:
\begin{align*}
\Prb[\cF_{n,N}^c]&\leq \sum_{x\in \cyl(NA,h(n)/2+d)\cap \sZ^d}\Prb[\card_v(\sC_{G,0}(x))\geq h(N)/4]\\
&\leq \card_v(\cyl(NA,h(n)/2+d)\cap\sZ^d)\Prb[\card_v(\sC_{G,0}(0))\geq h(N)/4]\\
&\leq c'_d \cH^{d-1}(NA)h(n)\kappa_1\exp(-\kappa_2 h(N)/4)
\end{align*}
where $c'_d$ is a constant depending only on the dimension $d$. We recall that by Theorem \ref{controlcluster}, $\E[|\partial_e\sC_{G,0}(0)|]<\infty$ and $\E[|\partial_e\sC_{G,0}(0)|^2]<\infty$ because $G(\{0\})>1-p_c(d)$. Moreover, as $h(N)/\log(N)$ goes to infinity when $N$ goes to infinity, the third term in the right hand side of \eqref{expchi} goes to $0$ when $N$ goes to infinity.
We now want to control the size of $V_0$. There exists a constant $c_d$ depending only on the dimension $d$ such that:
$$\card_v(V^1_0)\leq c_d \frac{\cH^{d-1}(NA)}{\cH^{d-1}(nA')}\cH^{d-2}(\partial(nA'))\,,$$
 $$\card_v(V^2_0)\leq c_d \cH^{d-2}(\partial(NA))\Diam(nA')$$
 and
 $$\card_v(V^3_0)\leq c_d \cH^{d-2}(\partial(NA))h(N)\, ,$$
thus 
\begin {align*}
\card_v(V_0)&\leq c_d \Bigg (\frac{\cH^{d-1}(NA)}{\cH^{d-1}(nA')}\cH^{d-2}(\partial(nA'))\\
& \hspace{3cm} +\cH^{d-2}(\partial(NA))(\Diam(nA')+h(N))\Bigg)
\end{align*}
and finally since $h(N)/N$ goes to $0$ as $N$ goes to infinity we obtain
$$\lim_{n\rightarrow \infty}\lim_{N\rightarrow \infty}\frac{\card_v(V_0)}{\cH^{d-1}(NA)}=0\,.$$
By first sending $N$ to infinity and then $n$ to infinity in inequality \eqref{expchi}, we get that
$$\limsup_{N\rightarrow \infty}\frac{\E[\chi_G(NA,h(N),\vv)]}{\cH^{d-1}(NA)}\leq \liminf_{n\rightarrow \infty}\frac{\E[\chi_G(nA',h'(n),\vv)]}{\cH^{d-1}(nA')}\, .$$

By setting $A=A'$ and $h=h'$, we deduce the existence of the limit $\lim_{n\rightarrow \infty}\E[\chi_G(nA,h(n),\vv)]/\cH^{d-1}(nA)$ and the inequality
$$\lim_{n\rightarrow \infty}\frac{\E[\chi_G(nA,h(n),\vv)]}{\cH^{d-1}(nA)}\leq \lim_{n\rightarrow \infty}\frac{\E[\chi_G(nA',h'(n),\vv)]}{\cH^{d-1}(nA')}\,.$$
Exchanging the role of $A, h$ and $A', h'$, we conclude that the two limits are equal. Note that $\chi_G$ does not depend on all the distribution $G$ but only on $G(\{0\})$. Indeed, let us couple $(t_G(e))_{e\in\E^d}$ with a family  $(\hat{t}(e))_{e\in\E^d}$ of Bernoulli of parameter $1-G(\{0\})$ in the following way: for an edge $e\in\E^d$, $\hat{t}(e)=\ind_{t_G(e)>0}$. With this coupling, the value of $\chi_G$ is the same for the two families of capacities. Therefore, the limit does not depend on $A$ nor $h$ but only on the direction $\vv$, on $G(\{0\})$ and on $d$, we denote it by $\zeta_{G(\{0\})}(\vv)$.
Thanks to inequality \eqref{control}, we know that there exists a constant $c'_d$ depending only on the dimension $d$ such that
\begin{align*}
\frac{\E[\chi_G(nA,h(n),\vv)]}{\cH^{d-1}(nA)}\leq c'_d \E[\card_v(\sC_{G,0}(0))]<\infty \,,
\end{align*}
thus $\zeta_{G(\{0\})}(\vv)$ is finite. 
\end{proof}

\section{From slabs to cylinders}\label{s5}
We recall that the quantity of interest is the flow through the cylinder, and that we have studied the flow from a thick rectangle to an hyperplane for technical reasons. In this section we are going to show that these flows are quite similar, more precisely we want to show the following proposition.

\begin{prop}\label{egalite} Let $G$ be a distribution on $[0,+\infty]$ such that $G(\{0\})>1-p_c(d)$. For any $\vv\in\sS^{d-1}$, for any non-degenerate hyperrectangle $A$ normal to $\vv$, for any height function $h$ that satisfies condition $(\star)$,
$$\lim_{n\rightarrow \infty}\frac{\E[\chi_G(nA,h(n),\vv)]}{\cH^{d-1}(nA)}=\lim_{n\rightarrow \infty}\frac{\E[\psi_G(nA,h(n),\vv)]}{\cH^{d-1}(nA)}=\zeta_{G(\{0\})}(\vv)\,.$$
\end{prop}

\begin{proof}
Let $A$ be a non-degenerate hyperrectangle and $h$ a height function satisfying condition $(\star)$. Let $\vv$ be one of the two unit vectors normal to $A$. We prove Proposition \ref{egalite} in two steps. In the first step, we obtain an upper bound for $\E[\chi_G(nA,h(n),\vv)]$ by building a cutset of null capacity between the top and the bottom of $\cyl(nA,h(n))$ from a cutset in $\slab(nA,h(n),\vv)$ that achieves the infimum in $\psi_G(nA,h(n),\vv)$. In the second step, we obtain a lower bound for $\E[\chi_G(nA,h(n),\vv)]$  by doing the reverse, \textit{i.e.}, we build a cutset between a translate of $\overline{nA}$ and $\hyp(nA+h(n)\vv)$, from a cutset in $\cyl(nA,h(n))$ that achieves the infimum in the definition of $\chi_G(nA,h(n),\vv)$.
\begin{figure}[H]
\def\svgwidth{1.1\textwidth}
\begin{center}
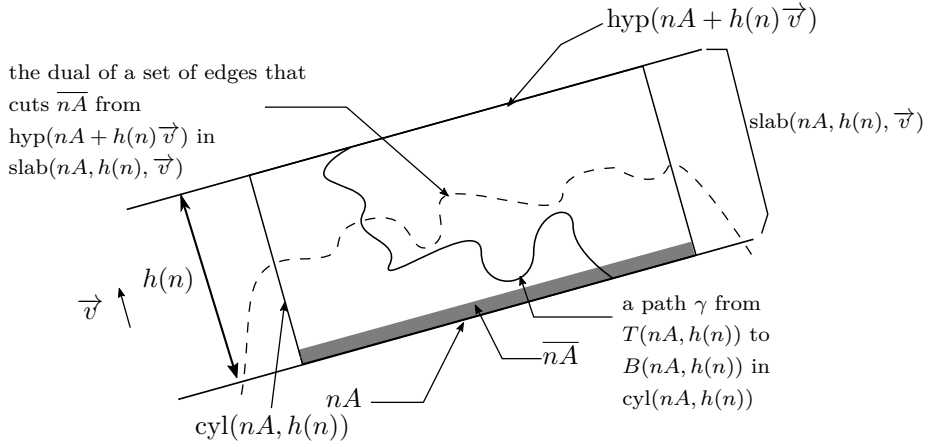
\caption[fig3]{\label{fig3}A cutset that cuts $\overline{nA}$ from $\hyp(nA+h(n)\vv)$ in $\slab(nA,h(n),\vv)$ and the top from the bottom of the cylinder $\cyl(nA,h(n))$ on the event $\cE_n$.}
\end{center}
\end{figure}
\textit{Step (i)}: We denote by $\cE_n$ the following event 
$$\cE_n=\bigcap_{x\in \cyl(nA,h(n)/2)\cap\sZ^d} \left\{\card_v(\sC_{G,0}(x))<\frac{h(n)}{2}\right\}\,.$$

On the event $\cE_n$, we have that $H_{G,h(n)}(nA)=h(n)$. By definition, we have $B(nA,h(n))\subset {\overline{nA}\cap\sZ^d}$ and $T(nA,h(n))\subset W(nA,h(n),\vv)$. 
On the event $\cE_n$, as any path from the top to the bottom of $\cyl(nA,h(n))$ is also a path from $\hyp(nA+h(n)\vv)$ to $\overline{nA}$ in $\slab(nA,h(n),\vv)$, any cutset that cuts $\hyp(nA+h(n)\vv)$ from $\overline{nA}$ is also a cutset from the top to the bottom in the cylinder (see Figure \ref{fig3}). Finally, any cutset that achieves the infimum in $\chi_G(nA,h(n),\vv)$ is a cutset of null capacity (and therefore of minimal capacity) for the flow from the top to the bottom in cylinder $\cyl(nA,h(n))$. 
Thus, on the event $\cE_n$, 
$$\psi_G(nA,h(n),\vv)\leq \chi_G(nA,h(n),\vv)\,.$$
Finally, for a constant $C_d$ depending only on the dimension $d$. 
\begin{align*}
\frac{\E[\psi_G(nA,h(n),\vv)]}{\cH^{d-1}(nA)}&\leq \frac{\E[\psi_G(nA,h(n),\vv)\mathds{1}_{\cE_n}]}{\cH^{d-1}(nA)}+\frac{\E[\psi_G(nA,h(n),\vv)\mathds{1}_{\cE_n^c}]}{\cH^{d-1}(nA)}\\
&\leq \frac{\E[\chi_G(nA,h(n),\vv)]}{\cH^{d-1}(nA)}+\frac{\card_e(\cyl(nA,h(n))\cap \E^d)\,\cdot\Prb(\cE_n^c)}{\cH^{d-1}(nA)}\\
&\leq \frac{\E[\chi_G(nA,h(n),\vv)]}{\cH^{d-1}(nA)}\\
&\hspace{2cm}+C_d h(n)^2\cH^{d-1}(nA)\kappa_1\exp(-\kappa_2 h(n)/2)\, ,
\end{align*}
where we use in the last inequality Theorem \ref{controlcluster}.
As $h$ satisfies condition $(\star)$ the second term of the right hand side goes to $0$ when $n$ goes to infinity and we obtain 
\begin{align}\label{limsup}
\limsup_{n\rightarrow \infty}\frac{\E[\psi_G(nA,h(n),\vv)]}{\cH^{d-1}(nA)}\leq \lim_{n\rightarrow \infty}\frac{\E[\chi_G(nA,h(n),\vv)]}{\cH^{d-1}(nA)}=\zeta_{G(\{0\})}(\vv)\,.
\end{align}

\textit{Step (ii)}: There exists an hyperrectangle $T'$, a small vector $\overrightarrow{t}$ and an integer vector $\overrightarrow{u}$ such that $T'=nA+\overrightarrow{u}$, $T'=nA-d\vv+\overrightarrow{t}$, $\|\overrightarrow{t}\|_\infty< 1$ and $\overrightarrow{t}\cdot \vv\leq 0$. Therefore, we have $-d-\sqrt{d}\leq \overrightarrow{u}\cdot \vv< -d$ and $\overline{T'}\subset \slab(A,\infty,\vv)^c$.
We now want to build a set of edges of null capacity that cuts $\overline{T'}$ from $\hyp(T' +H_{G,h(n)-\overrightarrow{u}\cdot \vv}(T')\vv)$ starting from a cutset between the top and the bottom of the cylinder $\cyl(nA,h(n))$. We define
$$\cE'_n=\bigcap_{x\in \cV(\cyl(nA,h(n)/2),2d)\cap\sZ^d} \left\{\card_v(\sC_{G,0}(x))<\frac{h(n)}{2}\right\}\,.$$
On the event $\cE'_n$, the minimal capacity of a cutset for the flow from the top to the bottom of the cylinder $\cyl(nA,h(n))$ is null (the set of null capacity  $\cup_{x\in\overline{nA}\cap \sZ^d}\partial_e\sC_{G,0}(x)$ is a cutset) and as the cylinder $\cyl(T',(h(n)-\overrightarrow{u}\cdot \vv)/2)$ is included in $\cV(\cyl(nA,h(n)/2),2d)$, we obtain $H_{G,h(n)-\overrightarrow{u}\cdot \vv}(T')=h(n)-\overrightarrow{u}\cdot \vv$ so that $\hyp(T'+H_{G,h(n)-\overrightarrow{u}\cdot \vv}(T')\vv)=\hyp(nA+h(n)\vv)$. We denote by $E$ one of the sets that achieve the infimum in $\psi_G(nA,h(n),\vv)$. In order to build a set that cuts $\overline{T'}$ from $\hyp(nA+h(n)\vv)$ from $E$, we need to add to $E$ edges to prevent flow from escaping through the vertical sides of $\cyl(nA,h(n))$. Let $V$ be a set that contains a discretized version of the vertical sides of $\cyl(nA,h(n))$. More precisely, we define by $V=\cV(\cyl(\partial(nA),h(n)),2d)\cap \sZ^d$. 

 
We aim to show that the following set 
$$F=E\cup \left(\bigcup_{x\in V}\partial_e\sC_{G,0}(x)\right)$$
cuts $\overline{T'}$ from $\hyp(nA+h(n)\vv)$ on the event $\cE'_n$ (see Figure \ref{fig4}). 


Let $\gamma=(y=v_0,e_1,v_1,\dots, e_m,v_m=x)$ be a path from $y\in W(nA,h(n),\vv)$ to $x\in \overline{T'}$ that stays in $\slab(T',h(n)-\overrightarrow{u}\cdot \vv,\vv)$. Let us consider the first moment $\gamma$ exits $\slab(nA,h(n),\vv)$, \textit{i.e.}, we define $p=\inf\{i\in\{0,\dots,m\}, \,v_j\notin \slab(nA,h(n),\vv)\}$.  We distinguish several cases. 

\begin{itemize}
\item Suppose that $v_{p-1}\in B(nA,h(n))$ and  $\gamma'=(v_0,e_1,\cdots,e_{p-1},v_{p-1})$, the portion of $\gamma$ between $v_0$ and $v_{p-1}$, stays in cylinder $\cyl(nA,h(n))$. Then $\gamma'$ is a path from the top to the bottom of $\cyl(nA,h(n))$ that stays in $\cyl(nA,h(n))$, thus $\gamma'\cap E \neq \emptyset$ and $\gamma\cap E\neq \emptyset$.
\item Suppose that $v_{p-1}\in B(nA,h(n))$ and that $\gamma'$ does not stay in the cylinder $\cyl(nA,h(n))$. Thus $\gamma'$ must intersect the boundary of the cylinder $\cyl(nA,h(n))$. As $\gamma'$ stays in $\slab(nA,h(n),\vv)$, $\gamma'$ can only intersect the vertical sides of the cylinder, \textit{i.e.}, $\cyl(\partial(nA),h(n))$, we obtain $\gamma'\cap V \neq \emptyset$. There exists $z\in V$ such that $\gamma'\cap \{z\}\neq \emptyset$. On the event $\cE'_n$, $\gamma'$ cannot be included in $\sC_{G,0}(z)$. Indeed, if $\gamma' \subset\sC_{G,0}(z)$, then $\gamma' \subset \sC_{G,0}(x)$ and $\sC_{G,0}(x)$ has a diameter at least $h(n)$, it is impossible on the event $\cE'_n$. Therefore we obtain $\gamma' \cap \partial_e \sC_{G,0}(z)\neq \emptyset$ and $\gamma\cap F \neq \emptyset$.
\item Suppose now that $v_{p-1}\notin B(nA,h(n))$, thus $v_{p-1}\notin \cyl(nA,h(n))$. If $x\in V$, we conclude as in the previous case that on the event $\cE'_n$, $\gamma \cap \partial_e\sC_{G,0}(x)\neq \emptyset$ and $\gamma \cap F \neq \emptyset$. If $x\notin V$, then $x\in \cyl(nA-2d \vv,h(n)+2d)$. As $v_{p-1}\notin \cyl(nA-2d \vv,h(n)+2d)$ and $\gamma$ stays in $\slab(nA-2d\vv, h(n)+2d,\vv)$, $\gamma$ cuts $\cyl (\partial(nA-2d\vv),h(n)+2d)$ and $\gamma\cap V \neq \emptyset$. We conclude as in the previous cases that $\gamma \cap F \neq \emptyset$.   

\end{itemize}

On the event $\cE'_n$, we obtain that $\gamma\cap F\neq \emptyset$. Moreover, the set $E$ has null capacity so it is also the case for the set $F$. Thus, the set $F$ cuts $\overline{T'}$ from $\hyp(nA+h(n)\vv)$ and has null capacity on the event $\cE'_n$.
\begin{figure}[H]
\def\svgwidth{1.2\textwidth}
\begin{center}
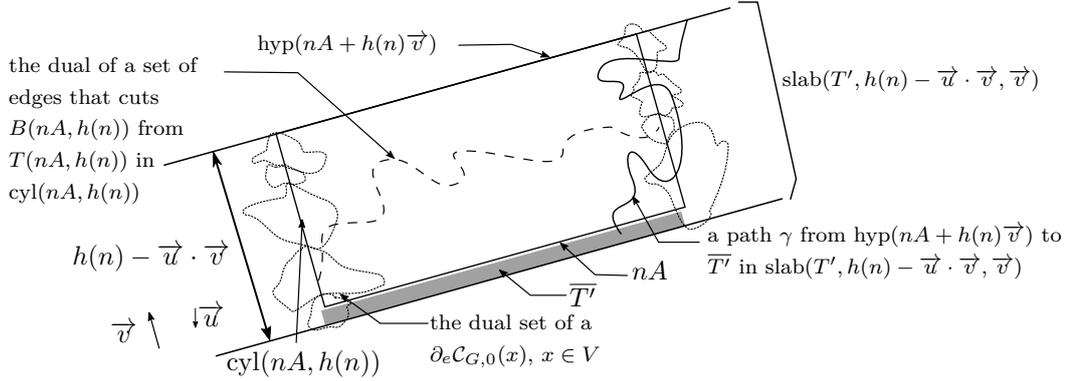
\caption[fig4]{\label{fig4}Construction of a cutset from $\overline{T'}$ to $\hyp(nA+h(n)\vv)$ from a cutset from the top to the bottom in the cylinder $\cyl(nA,h(n))$ on the event $\cE'_n$}
\end{center}
\end{figure}
Finally, for a constant $C'_d$ depending on $d$,

\begin{align}\label{ineqret}
&\frac{\E[\chi_G(T',h(n)-\overrightarrow{u}\cdot \vv,\vv)]}{\cH^{d-1}(nA)}\nonumber\\
&\hspace{0.8cm}\leq \frac{\E[\chi_G(T',h(n)-\overrightarrow{u}\cdot \vv,\vv)\mathds{1}_{\cE'_n}]}{\cH^{d-1}(nA)}+\frac{\E[\chi_G(T',h(n)-\overrightarrow{u}\cdot \vv,\vv)\mathds{1}_{\cE_n^{'c}}]}{\cH^{d-1}(nA)}\nonumber\\
&\hspace{0.8cm}\leq \frac{\E[|F|]}{\cH^{d-1}(nA)}+\frac{\E[\sum_{x\in \overline{T'}\cap\sZ^d}|\partial_e(\sC_{G,0}(x))|\mathds{1}_{\cE_n^{'c}}]}{\cH^{d-1}(nA)}\nonumber\\
&\hspace{0.8cm}\leq \frac{\E[\psi_G(nA,h(n),\vv)]+C'_d h(n)\cH^{d-2}(\partial(nA))\E[|\partial_e(\sC_{G,0}(0))|]}{\cH^{d-1}(nA)}\nonumber\\
&\hspace{1.4cm}+\frac{C'_d\cH^{d-1}(T')\sqrt{\E[|\partial_e(\sC_{G,0}(0))|^2]C_d h(n)\cH^{d-1}(nA)\kappa_1\exp(-\kappa_2 h(n))}}{\cH^{d-1}(nA)}\nonumber\\
&\hspace{0.8cm}\leq \frac{\E[\psi_G(nA,h(n),\vv)]}{\cH^{d-1}(nA)}+\frac{C'_d h(n)\cH^{d-2}(\partial(nA))\E[|\partial_e(\sC_{G,0}(0))|]}{\cH^{d-1}(nA)}\nonumber\\
&\hspace{1.4cm}+C''_d\sqrt{\cH^{d-1}(nA)\E[|\partial_e(\sC_{G,0}(0))|^2]h(n)\kappa_1\exp(-\kappa_2 h(n))}
\end{align}
where we use in the second inequality the control of $\chi_G(T',h(n)-\overrightarrow{u}\cdot \vv,\vv)$ obtained in Lemma \ref{lem1} and Cauchy-Scwartz' inequality in the third inequality.

As $\cH^{d-1}(nA)$ is of order $n^{d-1}$, $\cH^{d-2}(\partial(nA))$ is of order $n^{d-2}$  and $h$ satisfies condition $(\star)$, the second and the third terms of the right hand side of the inequality \eqref{ineqret} go to $0$ as $n$ goes to infinity. Moreover, thanks to Proposition \ref{subb}, using the  invariance of the model by the translation by an integer vector and the fact that the limit $\zeta_{G(\{0\})}(\vv)$ does not depend on the height function,
\begin{align*}
\lim_{n\rightarrow\infty}\frac{\E[\chi_G(T',h(n)-\overrightarrow{u}\cdot \vv,\vv)]}{\cH^{d-1}(T')}&=\lim_{n\rightarrow\infty}\frac{\E[\chi_G(nA,h(n)-\overrightarrow{u}\cdot \vv,\vv)]}{\cH^{d-1}(T')}\\
&=\lim_{n\rightarrow\infty}\frac{\E[\chi_G(nA,h(n),\vv)]}{\cH^{d-1}(nA)}\\
&=\zeta_{G(\{0\})}(\vv)\,.
\end{align*}
Thus, we obtain from \eqref{ineqret}
\begin{align}\label{liminf}
\zeta_{G(\{0\})}(\vv)=\lim_{n\rightarrow\infty}\frac{\E[\chi_G(nA,h(n),\vv)]}{\cH^{d-1}(nA)}\leq \liminf_{n\rightarrow\infty} \frac{\E[\psi_G(nA,h(n),\vv)]}{\cH^{d-1}(nA)}\,.
\end{align}
Combining inequalities \eqref{limsup} and \eqref{liminf}, we get that
\begin{align*}
\lim_{n\rightarrow\infty}\frac{\E[\chi_G(nA,h(n),\vv)]}{\cH^{d-1}(nA)}=\lim_{n\rightarrow\infty}\frac{\E[\psi_G(nA,h(n),\vv)]}{\cH^{d-1}(nA)}=\zeta_{G(\{0\})}(\vv)\,.
\end{align*}

\end{proof}

\section{Concentration}\label{s6}
 We aim here to prove Theorem \ref{heart}. To prove this theorem, we will need Proposition \ref{egalite} and the concentration inequality stated in Proposition \ref{ESI} that gives a control of the variance of a function of a finite family of independent random variables. We are not going to apply Proposition \ref{ESI} directly to $\psi_G$ for general distributions $G$ on $[0,+\infty]$ because Proposition \ref{ESI} does not apply \textit{a priori} for random variables than can take infinite value. Although we believe that this proposition may be adapted to distributions on $[0,+\infty]$, we prefer here to apply this proposition to $\psi_{G'}$ for simpler distributions $G'=G(\{0\})\delta_0 +(1-G(\{0\}))\delta_1$. The idea behind is that we are interested only in edges of null passage times because when the cylinder is big enough there exists almost surely a null cutset and so $\psi_G$ and $\psi_G'$ properly renormalized converges towards the same almost sure limit. More precisely, we need to show the following Lemma to prove Theorem \ref{heart}.
\begin{lem}\label{lem19}Let $p<p_c(d)$, we define $G_p=p\delta_1+ (1-p)\delta_0$. For every function $h$ satisfying condition $(\star)$, for every $\vv\in\sS^{d-1}$, for every non-degenerate hyperrectangle $A$ normal to $\vv$,  
$$\lim_{n\rightarrow \infty}\frac{\psi_{G_p}(nA,h(n),\vv)}{\cH^{d-1}(nA)}=\zeta_{1-p}(\vv) \text{ a.s..}$$\end{lem}

\begin{rk}
The advantage of using a concentration inequality on $\psi_G$ rather than on $\chi_G$ is that $\psi_G$ depends on the capacity of a finite deterministic set of edges whereas $\chi_G$ depends on an infinite set of edges (the edges in $\slab(A,\infty,\vv)$). Therefore $\psi_G$ is more appropriate to apply this concentration inequality.
\end{rk}

\begin{proof}[Proof of Lemma \ref{lem19}]
 Let $p<p_c(d)$. Let $\vv\in\sS^{d-1}$. Let $A$ be a non-degenerate hyperrectangle normal to $\vv$ and $h$ an height function that satisfies condition $(\star)$.
 We consider the cylinder $\cyl(nA,h(n))$ and we enumerate its edges as $e_1,\dots, e_{m_n}$. We define $(t_{G_p}(e_1),\dots, t_{G_p}(e_{m_n}),t'_{G_p}(e_1),\dots, t'_{G_p}(e_{m_n}))$ a family of independent random variables distributed according to distribution $G_p$. The quantity $\psi_{G_p}(nA,h(n),\vv)$ is a random variable that depends only on the capacities of the edges $e_1,\dots,e_{m_n}$. We define  $X=(t_{G_p}(e_1),\dots, t_{G_p}(e_{m_n}))$, $X^{(i)}=(t_{G_p}(e_1),\dots,t'_{G_p}(e_i),\dots, t_{G_p}(e_{m_n})$ for $i\in\{1,\dots,m_n\}$ and $f$ the function defined by $\psi_{G_p}(nA,h(n),\vv)=f(X)$. We denote by $\cG_n$ the following event that depends on $t_{G_p}(e_1),\dots,t_{G_p}(e_{m_n})$,
 $$\cG_n=\bigcap_{x\in\cyl(nA,h(n))\cap\sZ^d}\left\{\card_v(\sC_{G,0}(x))\leq \min\left(\frac{h(n)}{4},n^{1/4}\right)\right\}\,.$$
 On the event $\cG_n$, the minimal capacity of a cutset from the top to the bottom of the cylinder is null. 
Let $i\in\{1,\dots,m_n\}$, let us assume that $f(X)<f(X^{(i)})$. If $t_{G_p}'(e_i)<t_{G_p}(e_i)$, then we have $t_{G_p}(e_i)=1$, $t_{G_p}'(e_i)=0$ and on the event $\cG_n$, there exists a cutset of null capacity $E$ (thus $E$ does not contain $e_i$) that achieves the infimum in $f(X)$. It is still a cutset of null capacity in $\cyl(nA,h(n))$ for the distribution $X^{(i)}$. Thus, we obtain the following contradiction $f(X^{(i)})\leq |E|=f(X)$, so if $f(X)<f(X^{(i)})$ then $t_{G_p}'(e_i)\geq t_{G_p}(e_i)$ on the event $\cG_n$. 

Let us now assume that $f(X)<f(X^{(i)})$, then we have $t_{G_p}(e_i)\leq t_{G_p}'(e_i)$ on the event $\cG_n$. Let us denote by $R_n$ the intersection of all the minimal cutsets that achieve the infimum in $\psi_{G_p}(nA,h(n),\vv)$. If $e_i\notin R_n$, then there exists a cutset $E$ that does not contain $e_i$ and that achieves the infimum in $\psi_{G_p}(nA,h(n),\vv)$, \textit{i.e.}, $f(X)=|E|$. On the event $\cG_n$, all the cutsets that achieve the infimum in $f(X)$ are of null capacity.  Since $E$ is a cutset of null capacity that does not contain the edge $e_i$, it is still a cutset of null capacity in $\cyl(nA,h(n))$ for the distribution $X^{(i)}$. Thus, $f(X^{(i)})\leq f(X)$, which is a contradiction. 
Thus on $\cG_n$, if $f(X)<f(X^{(i)})$ then $t_{G_p}(e_i)\leq t'_{G_p}(e_i)$ and $e_i\in R_n$. We denote by $E$ a cutset that achieves the infimum in $f(X)$. We have $e_i\in E$, let us define $\widetilde{E}=(E\cup \partial_e\sC_{G,0}(e_i^+)\cup \partial_e\sC_{G,0}(e_i^-))\setminus\{e_i\}$ where we write $e_i=\langle e_i^-,e_i^+\rangle$. This set has null capacity for both distributions $X$ and $X^{(i)}$. Let us prove that on the event $\cG_n$, the set $\widetilde{E}$ cuts the top from the bottom of cylinder $\cyl(nA,h(n))$.
 
 Let $\gamma$ be a path from $x\in T(nA,h(n))$ to $y\in B(nA,h(n))$. If $e_i\notin \gamma$ then as $E$ is a cutset, we have that $\gamma\cap E\setminus\{e_i\} \neq \emptyset$ thus $\gamma\cap \widetilde{E}\neq \emptyset$. We now assume that  $e_i\in \gamma$. On the event $\cG_n$, $\gamma$ cannot be included in $\sC_{G,0}(e_i^+)\cup \sC_{G,0}(e_i^-)$. Thus, either $x\notin \sC_{G,0}(e_i^+)\cup \sC_{G,0}(e_i^-)$ or $y\notin \sC_{G,0}(e_i^+)\cup \sC_{G,0}(e_i^-)$. We study only the case $y\notin \sC_{G,0}(e_i^+)\cup \sC_{G,0}(e_i^-)$, the other case is studied similarly. We denote by $f$ the edge $\gamma$ takes to finally exit $\sC_{G,0}(e_i^+)\cup \sC_{G,0}(e_i^-)$, \textit{i.e.}, if we write $\gamma=(v_0,e'_1,v_1,\dots,e'_m,v_m)$ and we denote by $p=\max\{j\, : \, v_j\in \sC_{G,0}(e_i^+)\cup \sC_{G,0}(e_i^-)\}$ then $f=e'_{p+1}$. By definition of $p$, we must have $f\neq e_i$ and $f\in\partial_e\sC_{G,0}(e_i^+)\cup \partial_e\sC_{G,0}(e_i^-)\setminus\{e_i\}$. As $f\in\widetilde{E}$, we finally obtain that $\gamma\cap\widetilde{E}\neq \emptyset$ and that on the event $\cG_n$, $\widetilde{E}$ is indeed a cutset in the cylinder of null capacity for the distribution $X^{(i)}$. 
 
Thus on the event $\cG_n$ and when $f(X)<f(X^{(i)})$, we have that $f(X^{(i)})\leq |\widetilde{E}|$ and for large $n$
\begin{align*}
 f(X^{(i)})-f(X)&\leq \card_e(\sC_{G,0}(e_i^+)\cup \sC_{G,0}(e_i^-))\\
 &\leq 2\min\left(\frac{h(n)}{4},n^{1/4}\right)\\
 &\leq 2 n^{1/4} \,.
 \end{align*}
We can now apply Proposition \ref{ESI} to obtain
 \begin{align}\label{eqESI}
 \Var(f(X))&\leq \sum_{i=1}^{m_n}\E\left[(f(X)-f(X^{(i)}))_-^2\right]\nonumber\\
 &\leq \sum_{i=1}^{m_n}\E\left[(f(X)-f(X^{(i)}))_-^2\ind_{\cG_n}\right]+\E\left[(f(X)-f(X^{(i)}))_-^2\ind_{\cG_n^c}\right]\nonumber\\
 &\leq \sum_{i=1}^{m_n}\E\left[(f(X)-f(X^{(i)}))_-^2\ind_{\cG_n,\,e_i\in R_n}\right]\nonumber\\
 &\hspace{1.5cm}+\card_e(\cyl(nA,h(n))\cap \E ^d)^2\Prb[\cG_n^c]\nonumber\\ 
 &\leq \sum_{i=1}^{m_n}\E\left[4\sqrt{n}\ind_{e_i\in R_n}\right]+\card_e(\cyl(nA,h(n))\cap \E^d)^2\Prb[\cG_n^c]\nonumber\\
  &\leq4\sqrt{n}\E\left[card_e(R_n)\right]+\card_e(\cyl(nA,h(n))\cap\E^d)^2 \nonumber\\
  &\hspace{1.5cm}\times \card_v(\cyl(nA,h(n))\cap \sZ^d) \kappa_1\exp(-\kappa_2 \min(h(n)/4,n^{1/4}))\nonumber\\
   &\leq 4\sqrt{n}\E\left[\psi_{G_p}(nA,h(n),\vv)\right]\nonumber\\
   &\hspace{1.5cm}+\card_e(\cyl(nA,h(n))\cap \E ^d)^3\kappa_1\exp(-\kappa_2 \min(h(n)/4,n^{1/4}))\,.
\end{align} 
Let $\ep>0$, by Bienaymé-Tchebytchev inequality,
\begin{align}\label{eqMarkov}
&\Prb\left[\left|\frac{\psi_{G_p}(nA,h(n),\vv)}{\cH^{d-1}(nA)}-\E\left[\frac{\psi_{G_p}(nA,h(n),\vv)}{\cH^{d-1}(nA)}\right]\right|>\ep\right]\leq \frac{\Var(f(X))}{(\ep\cH^{d-1}(nA))^2}\,.
\end{align}   
As $\cH^{d-1}(nA)$ and $\E\left[\psi_{G_p}(nA,h(n),\vv)\right]$ are both of order $n^{d-1}$, by inequalities \eqref{eqESI} and \eqref{eqMarkov} and since $h$ satisfies $(\star)$, we obtain that 
\begin{align}\label{concentration}
\Prb\left[\left|\frac{\psi_{G_p}(nA,h(n),\vv)}{\cH^{d-1}(nA)}-\E\left[\frac{\psi_{G_p}(nA,h(n),\vv)}{\cH^{d-1}(nA)}\right]\right|>\ep\right]= O\left(\dfrac{1}{n^{d-3/2}}\right)\,.
\end{align} 
We can conclude that for $d\geq 3$, the sum $$\sum_{n=1}^\infty\Prb\left[\left|\frac{\psi_{G_p}(nA,h(n),\vv)}{\cH^{d-1}(nA)}-\E\left[\frac{\psi_{G_p}(nA,h(n),\vv)}{\cH^{d-1}(nA)}\right]\right|>\ep\right]$$ is finite. By Borel-Cantelli Lemma, we deduce that almost surely $$\limsup_{n\rightarrow\infty}\left|\frac{\psi_{G_p}(nA,h(n),\vv)}{\cH^{d-1}(nA)}-\E\left[\frac{\psi_{G_p}(nA,h(n),\vv)}{\cH^{d-1}(nA)}\right]\right|\leq 0\,,$$ and finally,
$$\lim_{n\rightarrow\infty}\frac{\psi_{G_p}(nA,h(n),\vv)}{\cH^{d-1}(nA)}=\zeta_{1-p}(\vv)\, \text{  a.s.}\,.$$
 \end{proof}
 
 \begin{proof}[Proof of Theorem \ref{heart}]
Let $G$ be a distribution on $[0,+\infty]$ such that $G(\{0\})>1-p_c(d)$. We define $p=1-G(\{0\})$, we have $p< p_c(d)$. Let $\vv\in\sS^{d-1}$. Let $A$ be a non-degenerate hyperrectangle normal to $\vv$ and $h$ an height function that satisfies condition $(\star)$. We couple the capacities $(t_G(e))_{e\in\E^d}$ and $(t_{G_p}(e))_{e\in\E^d}$ in the following way: for $e\in \E^d$, $t_{G_p}(e)=\ind_{t_G(e)>0}$. Note that for any $e\in \E^d$, $t_G(e)=0$ if and only if $t_{G_p}(e)=0$. We define the event $\cH_n$,
$$\cH_n=\bigcap_{x\in B(nA,h(n))}\left \{ \card_v(\sC_{G,0}(x))\leq \frac{h(n)}{2}\right\}\,.$$
On the event $\cH_n$, there exists a cutset of null capacity between the top and the bottom of $\cyl(nA,h(n))$, and the cutsets of null capacity that achieve the infimum in $\psi_G(nA,h(n),\vv)$ and $\psi_{G_p}(nA,h(n),\vv)$ are the same. Thus, on the event $\cH_n$,  $\psi_G(nA,h(n),\vv)=\psi_{G_p}(nA,h(n),\vv)$. By Theorem \ref{controlcluster}, there exists a constant $c_d$ depending only on $d$ such that $$\Prb[\cH_n^c]\leq c_d \cH^{d-1}(nA) \kappa_1\exp(-\kappa_2 h(n)/2)\,.$$ As $h$ satisfies condition $(\star)$, we conclude that
$\sum_n \Prb[\cH_n^c]<\infty\, .$
By Borel-Cantelli Lemma and by Lemma \ref{lem19}, we deduce that
$$\lim_{n\rightarrow\infty}\frac{\psi_{G}(nA,h(n),\vv)}{\cH^{d-1}(nA)}=\lim_{n\rightarrow\infty}\frac{\psi_{G_p}(nA,h(n),\vv)}{\cH^{d-1}(nA)}=\zeta_{G(\{0\})}(\vv)\, \text{  a.s.}\,.$$
\end{proof}
 \begin{rk}
Unfortunately at this stage, we are not able to prove this result for $d=2$. The concentration's argument we use in this section fails in dimension $2$ (see \eqref{concentration}).

 \end{rk}

\bibliographystyle{plain}
\bibliography{biblio}

\end{document}